\documentclass[11pt,reqno]{amsart}
\usepackage[english]{babel}
\usepackage{color}
\usepackage{amsmath, latexsym, amsthm, amsfonts,bm,amssymb,verbatim} 
\usepackage{caption}
\usepackage{natbib}
\usepackage[text={5.8in,8.0in},centering]{geometry}
\usepackage{epsfig}
\usepackage{graphicx}
\usepackage[colorlinks,linkcolor=blue, citecolor=blue,urlcolor=blue]{hyperref}
\usepackage[dvipsnames]{xcolor}
\usepackage[normalem]{ulem}
\newcommand{\dd}{{\,\mathrm d}}
\newcommand{\si}{\sigma}

\renewcommand{\th}{\theta}
\newcommand{\la}{\lambda}
\newcommand{\ga}{\gamma}

\newcommand{\eps}{\varepsilon}

\renewcommand{\phi}{\varphi}

\newcommand{\scr}[1]{{\mathcal #1}}

\newcommand{\EE}{\mathbb{E}}

\newcommand{\PP}{\mathbb{P}}
\newcommand{\ind}{\mathbf{1}}
\newcommand{\RR}{\mathbb{R}}
\newcommand{\ZZ}{\mathbb{Z}}

\newcommand{\NN}{\text{I\!N}}

\newcommand{\bem}{{\begin{bmatrix}}}
\newcommand{\enm}{{\end{bmatrix}}}

\newcommand{\T}{{\prime}}
\newcommand{\tr}{\operatorname{tr}}
\newcommand{\Col}{\operatorname{Col}}
 \newcommand{\DD}{{\,\mathrm D}}

\usepackage{bbm}
\newcommand{\R}{\RR}

 \newcommand{\rank}{\operatorname{rank}}

\theoremstyle{definition}
\newtheorem{thm}{Theorem}
\newtheorem{prop}[thm]{Proposition}
\newtheorem{lem}[thm]{Lemma}
\newtheorem{cor}[thm]{Corollary}
\newtheorem{rem}[thm]{Remark}
\newtheorem{ex}[thm]{Example}
\newtheorem{defn}[thm]{Definition}
\newtheorem{ass}[thm]{Assumption}

\numberwithin{thm}{section}

\newcommand{\expec}[1]{\ensuremath{{\rm E}\mspace{-1mu}\left[#1\right]}}

\newcommand{\rara}[1]{\renewcommand{\arraystretch}{#1}}

\newcommand{\e}{\mathrm{e}}

\newcommand{\di}{\triangle}

\newcommand{\Bm}{\begin{bmatrix}}
\newcommand{\Em}{\end{bmatrix}}
\newcommand{\lmin}{\lambda_{\rm min}}
\newcommand{\lmax}{\lambda_{\rm max}}

\newcommand{\Md}{M^\dagger}

\newcommand{\ZD}[1]{Z_{\Delta,#1}}
\newcommand{\LD}{L_\Delta}
\newcommand{\MD}{M_\Delta}

\graphicspath{{../}}

\begin{document}

\title[Simulation of conditional diffusions]{Simulation of elliptic and hypo-elliptic conditional diffusions}

\author[J.B.~Bierkens]{Joris Bierkens}
\address{Delft Institute of Applied Mathematics\\
Faculty of Electrical Engineering, Mathematics and Computer Science\\
Delft University of Technology\\
Van Mourik Broekmanweg 6\\
 2628 XE Delft\\
The Netherlands}
\email{joris.bierkens@tudelft.nl}

\author[F.H.~van der Meulen]{Frank van der Meulen}
\address{Delft Institute of Applied Mathematics\\
Faculty of Electrical Engineering, Mathematics and Computer Science\\
Delft University of Technology\\
Van Mourik Broekmanweg 6\\
 2628 XE Delft\\
The Netherlands}
\email{f.h.vandermeulen@tudelft.nl}

\author[M.~Schauer]{Moritz Schauer}
\address{Department of Mathematical Sciences\\
		Chalmers University of Technology and University of Gothenburg\\
		SE-412 96 G\"{o}teborg\\
		Sweden}
\email{smoritz@chalmers.se}



\begin{abstract}
Suppose $X$ is a multidimensional diffusion process. Assume that at time zero the state of $X$ is fully observed, but at time $T>0$ only linear combinations of its components are observed. That is, one only observes  the vector $L X_T$ for a given matrix $L$.  In this paper we show how samples from the conditioned process can be generated. The main contribution of this paper is to prove that  guided proposals, introduced in (\cite{Schauer}), can be used in a unified way for both uniformly and hypo-elliptic diffusions, also when $L$ is not the identity matrix.   This is illustrated by excellent performance in two challenging cases: a partially observed twice integrated diffusion with multiple wells and the partially observed FitzHugh-Nagumo model.
\end{abstract}

\date{\today}

\maketitle

\section{Introduction}
\noindent 

Let $X=(X_t,\, t\in [0, T])$ be a $d$-dimensional diffusion process satisfying the stochastic differential equation (SDE)
\begin{equation}\label{eq:sde} \dd X_t = b(t,X_t) \dd t + \si(t,X_t) \dd W_t,\qquad X_0=x_0,\qquad t\in [0,T]. \end{equation}
Here $b\colon [0,\infty)\times \RR^d \to \RR^d$, $\si\colon [0,\infty)\times \RR^d \to \RR^{d\times d'}$ and $W$ is a $d'$-dimensional Wiener process with all components independent. Stochastic differential equations are widely used for modelling in engineering, finance and biology, to name a few fields of applications. 
In this paper we will not only consider {\it uniformly elliptic} models, where it is assumed that there exists an $\eps >0$ such that for all $(t,x) \in [0,T] \times \RR^d$ and $y\in \RR^d$ we have $\|\si(t,x)^\T y\|^2 \ge \eps \|y\|^2$, but also hypo-elliptic models. These are models where the randomness spreads through all components - ensuring the existence of smooth transition densities of the diffusion, even though the diffusion is possibly not uniformly elliptic (for example because the Wiener noise only affects certain components.)
Such models appear frequently in  application areas; many examples are given in the  introductory section of \cite{clairon2017}.  A rich subclass of nonlinear hypo-elliptic diffusions that is included in our setup is specified by a drift of the form 
\begin{equation}\label{eq:descr-hyp1}
	b(t,x) = B x + \beta(t,x), 
\end{equation} 
where
\begin{equation}\label{eq:descr-hyp2} B=\Bm 0_{k\times k'} & I_{k\times k} \\ 0_{k'\times k'} & 0_{k'\times k}\Em \qquad \beta(t,x)=\Bm 0_{k\times 1} \\ \underline{\beta}(t,x)\Em \qquad  \si =\Bm 0_{k\times d'} \\ \underline{\sigma}(t) \Em \end{equation}
and $\underline{\si} \colon [0,\infty) \to  \RR^{k'\times d'}$, $\underline{\beta}\colon[0,\infty) \times \RR^d \to \RR^{k'}$ and  $k+k'=d$. This  includes several forms of integrated diffusions. 

Suppose $L$ is a $m\times d$ matrix with $m \le d$. We aim to simulate the process $X$, conditioned on the random variable
\[ V = L X_T. \]
The conditional process is termed a diffusion bridge, albeit its paths do not necessarily end at a fixed point but in the set $\{x\colon V = L x\}$.
 Besides being an interesting mathema\-tical problem on its own, simulation of such diffusion bridges is key to parameter estimation of diffusions from discrete observations. If the process is observed at discrete times directly or through an observation operator $L$, data-augmentation is routinely used for performing Bayesian inference (see for instance \cite{GolightlyWilkinson},  \cite{PapaRobertsStramer} and \cite{vdm-s-estpaper}). Here, a key step consists of the simulation of the ``missing'' data, which amounts to simulation of diffusion bridges. 
 
 Another application is nonlinear filtering, where  at time $t$ the state $X_t$ was observed and at time ${t+T}$ a new observation $L X_{t+T}$ comes in. Interest then lies in sampling from the distribution on $X_{t+T}$, conditional on $(X_t, L X_{t+T})$. The simulation method developed in this paper can then be used for constructing efficient particle filters. We leave the application of our methods to estimation and filtering to future research, although it is clear that our results can be used directly  within the algorithms given in \cite{vdm-s-estpaper}. Finally, rare-event simulation is a third application area for which our results are useful. 

We aim for a unified approach, by which we mean that the bridge simulation method applies simultaneously to uniformly elliptic and hypo-elliptic models. This is important, as in the aforementioned estimation problems either one of the two types of ellipticity may apply to the data.  While the sample paths of uniformly- and hypo-elliptic diffusions are very different, the  corresponding distributions of the observations can be very similar if the diffusion coefficients are close. Algorithms which are invalid for hypo-elliptic diffusions will therefore be numerically unstable if the model is close to being hypo-elliptic, and it may be a priori unknown if this is the case.

\subsection{Literature review} In case the diffusion is uniformly elliptic and the endpoint is fully observed, i.e.\ $L=I$, the problem has been studied extensively.  Cf.\  \cite{Clark}, \cite{DurhamGallant}, \cite{BeskosPapaspiliopoulosRobertsFearnhead}, \cite{DelyonHu}, \cite{beskos-mcmc-methods}, \cite{HaiStuaVo09}, \cite{LinChenMykland}), \cite{lindstrom}, \cite{Schoenmakers}, \cite{Bladt2}, \cite{Schauer} and \cite{Whitaker}.

 Much less is is known when either $L\neq I$ or   when the diffusion is not assumed to be uniformly elliptic. 
 In \cite{beskos-mcmc-methods} and \cite{HaiStuaVo09} a Langevin MCMC sampler is constructed for  sampling diffusion bridges when the drift is of the form $b(x) =B x + \si \si^\T \nabla V(x)$ and $\si$ is constant, assuming uniform ellipticity. Subsequently, this approach  was extended to hypo-elliptic diffusions of the form \[ \Bm X_t \\ Y_t \Em = \Bm Y_t \\ f(X_t) - Y_t\Em \dd t + \Bm 0 \\ 1\Em \dd W_t \]
 in \cite{Hairer2011}. However, no simulation results were included to the paper as ``these simulations proved prohibitively slow and the resulting method does not seem like a useful approach to sampling'' (\cite{Hairer2011}, page 671).

We will shortly review in more detail the works \cite{DelyonHu}, \cite{Marchand} and \cite{vdm-schauer}, as the present work builds upon these. The first of these papers includes some forms of hypo-elliptic diffusions, whereas the latter two papers consider uniformly elliptic diffusions with $L\neq I$. 

\cite{StramerRobertsJTSA} consider Bayesian estimation of nonlinear continuous-time autoregressive (NLCAR) processes using a data-augmentation scheme. This is a specific  class of hypo-elliptic models included by the specification \eqref{eq:descr-hyp1}--\eqref{eq:descr-hyp2}. The method of imputation is however different from what is proposed in this paper. 

\emph{Estimation} of discretely observed hypo-elliptic diffusions has been an active field over the past 10 years. As we stated before, within the Bayesian approach a data-augmentation strategy where diffusion bridges are imputed is natural. However, this is by no means the only way for doing estimation. Frequentist approaches to the estimation problem include \cite{Sorensen2012}, \cite{ditlevsen-samson2017}, \cite{lu2016}, \cite{comte2017}, \cite{SAMSON2012},  \cite{PokernStuartWiberg2009},   \cite{clairon2017} and \cite{melnykova}.

\subsection{Review of \cite{DelyonHu} and \cite{Schauer}}
To motivate and explain our approach, it is  useful to review shortly the methods developed in \cite{DelyonHu} and \cite{Schauer}. The method  that we propose in this article builds up on these papers. Both of these are restricted to the setting  $L=I$ (full observation of the diffusion at time $T$) and uniform ellipticity. Their common starting point  is that under mild conditions the diffusion bridge, obtained by conditioning on $L X_t = v$, is a diffusion process itself, governed by the SDE
\begin{equation}\label{Xstarequation} \dd X^\star_t = \left(b(t,X^\star_t) + a(t,X^\star_t)r(t,X^\star_t) \right) \dd t + \si(t,X^\star_t) \dd W_t,\qquad X^\star_0=x_0. \end{equation}
Here $a=\si \si^\T$ and $r(t,x) = \nabla_x \log p(t,x; T,v)$. We implicitly have assumed the existence of transition densities $p$ such that $\PP^{(t,x)}(X_T \in A)=\int_A p(t,x; T,\xi) \dd \xi$ and that $r(t,x)$ is well defined. The SDE for $X^\star$ can be derived from either Doob's $h$-transform or the theory of initial enlargement of the filtration. Unfortunately, the ``guiding'' term $a(t,X^\star_t)r(t,X^\star_t)$ appearing in the drift of $X^\star$ is intractable, as the transition densities $p$ are  not available in closed form. 
Henceforth, as direct simulation of $X^\star$ is infeasible, a common feature of both \cite{DelyonHu} and \cite{Schauer} is  to simulate a tractable process $X^\circ$ instead of $X^\star$, that resembles $X^\star$. Next,  the mismatch  can be  corrected for by a Metropolis-Hastings step or weighting. The {\it proposal} $X^\circ$ (the terminology is inherited from $X^\circ$ being a proposal for a Metropolis-Hastings step) is assumed to solve the SDE
\begin{equation}\label{eq:prop-sde}
\dd X^\circ_t = b^\circ(t, X^\circ_t) \dd t + \sigma(t, X^\circ_t) \dd W_t, \quad X^\circ_0 = x_0,  \quad t\in [0,T],
\end{equation}
where the drift $b^\circ$ is chosen such that the process $ X_t^\circ$ hits the correct endpoint (say $v$) at the final time $T$. \cite{DelyonHu}   proposed to take \begin{equation}\label{eq:DH}b^\circ(t,x) = \la b(t,x) + (v-x)/(T-t),\end{equation} where either $\la=0$ or $\la=1$,  the choice $\la=1$ requiring the drift $b$ to be bounded. If $\la=0$, a popular discretisation of this SDE is the Modified Diffusion Bridge introduced by \cite{DurhamGallant}. A drawback of this method is that the drift is not taken into account. 
In \cite{Schauer}  it was proposed to take 
\begin{equation}\label{eq:guideddrift}
			b^\circ(t, x) = b(t,x) + a(t, x) \tilde{r}(t,x)
\end{equation}
Here $\tilde{r}(t,x)=\nabla_x \log \tilde{p}(t,x; T,v)$, where $\tilde p(t,x)$ is  the transition density  of an auxiliary diffusion process $\tilde{X}$ that has tractable transition densities. In this paper, we always assume $\tilde{X}$ to be a linear process, i.e.\ a diffusion satisfying the SDE
\begin{equation}
	\label{eq:auxlin}
	\dd \tilde{X}_t =\tilde{b}(t,\tilde{X}_t) \dd t + \tilde\sigma(t) \dd W_t, \quad\text{where}\quad  \tilde{b}(t,x)=\tilde{B}(t) x +\tilde\beta(t).
\end{equation} 
The process $X^\circ$ obtained in this way will be referred to as a {\it guided proposal}. 

We denote the laws of $X$, $X^\star$ and $X^\circ$ viewed as measures on the space $C([0,t], \RR^d)$ of continuous functions from $[0,t]$ to $\RR^d$ equipped with its Borel-$\sigma$-algebra by  $\PP_t$, $\PP^\star_t$ and $\PP^\circ_t$ respectively.   \cite{DelyonHu} provided sufficient conditions such that $\PP^\star_T$ is absolutely continuous with respect to $\PP^\circ_T$ for the proposals derived from \eqref{eq:DH}. Moreover, closed form expressions for the Radon-Nikodym derivative were derived. For the proposals derived from \eqref{eq:guideddrift}, it was proved in \cite{Schauer} that the condition $\tilde\sigma(T)^\T \tilde\sigma(T) = a(T,v)$ is necessary for absolute continuity of $\PP^\star_T$ with respect to $\PP^\circ_T$. We refer to this condition as the {\it matching condition}, as the diffusivity of $X$ and $\tilde X$ need to match at the conditioning point. Under that condition (and some additional technical conditions), it was derived that 
  \[ \frac{\dd \PP^\star_T}{\dd \PP^\circ_T}(X^\circ) =\frac{\tilde{p}(0,x_0;T,v)}{p(0,x_0, T,v)} \exp\left(\int_0^T \scr{G}(s,X^\circ_s) \dd s\right), \]
  where $\scr{G}(s,x)$ is tractable. A great deal of work in the proof is concerned with proving that $\|X^\circ_t -v\| \to 0$ at the ``correct'' rate. 
  
\subsection{Approach}
 We aim to extend the results in  \cite{DelyonHu} and \cite{Schauer} by lifting the restrictions of
\begin{enumerate}
  \item uniform ellipticity;
  \item  $L$ being the identity matrix.
\end{enumerate}

\subsubsection{Extending \cite{DelyonHu}}
We  first explain the difficulty in extending this approach  beyond uniform ellipticity. To see the problem, we fix $t<T$. Absolute continuity of  $\PP^\star_t$ with respect to $\PP^\circ$ requires the existence of a mapping
$\eta(s,x)$ such that 
\begin{equation}\label{eq:eta} \si(s,x) \eta(s,x) = b^\star(s,x)-b^\circ(s,x),\qquad s\in [0,t], \quad x\in \RR^d,  \end{equation}
which follows from    Girsanov's theorem (\cite{LiptserShiryayevI}, Section 7.6.4). 
However, for the choice of \cite{DelyonHu} (as given in equation \eqref{eq:DH})  this mapping $\eta$ need not exist both in case $\la=0$ and $\la=1$. If  $\la=1$ then we have \[b^\star(s,x)-b^\circ(s,x)=\frac{v-x}{T-s},\] and therefore $\eta(s,x)$ only exists if $v-x$ is  in the column space of $\si(s,x)$. A similar argument applies to the case $\la=0$. From these considerations, it is not surprising that \cite{DelyonHu} need additional assumptions on the form of the drift to deal with the hypo-elliptic case. 
More specifically, they consider 
\begin{equation} \label{eq:delyonhu-hypoelliptic} \dd X_t = \left( \si(t) h(t,X_t) + B(t) x + \beta(t)\right) \dd t + \sigma(t) \dd W_t, \end{equation}
 with $\si(t)$ admitting  a left-inverse. Then they show that bridges can be obtained by simulating bridges corresponding to this SDE with $h\equiv 0$, followed by correcting for the discrepancy by weighting according to their likelihood ratio.
 Clearly, the form of the drift in the preceding display is restrictive, but necessary for absolute continuity. 
 
 Whereas lifting the assumption of uniform ellipticity seems hard, lifting the assumption that $L=I$ is possible. Indeed, it was shown by  \cite{Marchand} in a clever way how this can be done by using  the guiding term
\begin{equation}\label{eq:pull-marchand} v(t,x):=a(t,x) L^\T (L a(t,x) L^\T)^{-1} \frac{v-Lx}{T-t} \end{equation}
to be superimposed on the drift of the original diffusion.  Hence, the proposal satisfies the SDE
\[ \dd X^\di_t = b(t,X^\di_t) \dd t + v(t,X^\di_t) \dd t + \si(t,X^\di_t) \dd W_t, \qquad X^\di_0 =x_0. \]
By  applying Ito's lemma to $(T-t)^{-1} (LX^\di(t)-v)$, followed  by the law of the iterated logarithm for Brownian motion, the rate at which  $L X^\di(t)$ converges to $v$ can be derived. Interestingly, the same guiding term as in \eqref{eq:pull-marchand} was used in a specific setting by \cite{Arnaudon2018}, where the guiding term was rewritten as $\si(t,x) (L\si(t,x))^+ (v-Lx)/(T-t)$, 
assuming that $L\si$ has linearly independent rows. Here $A^+$ denotes the Moore-Penrose inverse of the matrix $A$. 
The form of the guiding term in \eqref{eq:pull-marchand} suggests that invertibility of $La(t,x)L^\T$ suffices, which, depending on the precise form of $L$, would allow for some forms of hypo-ellipticity. However, we believe  there are fundamental problems when one wants to include for example  integrated diffusions. We return to this in the discussion in section \ref{sec:discussion}.

 \subsubsection{Extending \cite{Schauer}}\label{subsubsec:extending-schauer}
  In case $L$ is not the identity matrix, the conditioned diffusion also satisfies the SDE \eqref{Xstarequation}, albeit with an adjusted definition of $r(t,x)$. To find the right form of $r(t,x)$, assume without loss of generality that $\rank L = m<d$.  
 Let $(f_1, \dots, f_m)$ denote an orthonormal basis of
 $\Col(L^\T)$, and let $(f_{m+1},\dots,f_d)$ denote an orthonormal basis of $\ker L$.
 Then for $A\subset \RR^m$
 \begin{multline*}   \PP^{(t,x)}\left(X_T \in A \times \RR^{d-m}\right) =\\ \int_A \left(\int_{\RR^{d-m}} p\left(t,x; T, \sum_{i=1}^d \xi_i f_i\right) \dd \xi_{m+1},\cdots ,\dd\xi_d \right) \dd\xi_1,\cdots,\dd\xi_m.
  \end{multline*}
Suppose $x=\sum_{i=1}^d \xi_i f_i$ is such that $Lx=v$. This is equivalent to
\begin{equation}\label{eq:fix-first-mxis}   \sum_{i=1}^m \xi_i L f_i = v, \end{equation}
since $f_{m+1},\ldots, f_d \in \ker L$. Hence if $\xi_1,\ldots, \xi_m$ are determined by \eqref{eq:fix-first-mxis} and if we define 
\[ \rho(t,x) = \int_{\RR^{d-m}} p\left(t,x; T, \sum_{i=1}^d \xi_i f_i\right) \dd \xi_{m+1},\cdots ,\dd\xi_d,  \]
then this is the density of $X_T \mid X_t$, concentrated on the subspace $LX_T=v$. 
 
In case $\rank L=d$, we can assume without loss of generality that $L=I$ which is the situation of fully observing $X_T$. Summarising, we  define
\begin{equation}\label{eq:def-rho} \rho(t,x)=\begin{cases} p(t,x; T, v) & \text{if}\quad m=d \\ \int_{\RR^{d-m}} p\left(t,x; T, \sum_{i=1}^d \xi_i f_i\right) \dd \xi_{m+1},\cdots ,\dd\xi_d\,   & \text{if}\quad m<d\end{cases}\end{equation}
and let 
 $r(t,x) =\nabla_x \rho(t,x)$. 
The definition of guided proposals  in the partially observed  hypo-elliptic case is then just as in the uniformly elliptic case with a full observation: replace the intractable transition density $p$ appearing in the definition of $\rho$ by $\tilde{p}$ to yield $\tilde\rho$. Then  define  \[\tilde{r}(t,x)=\nabla_x \log \tilde\rho(t,x)\] and let  the process  $X^\circ$ be  defined by equation \eqref{eq:prop-sde} with  $b^\circ(t,x)=b(t,x) + a(t,x) \tilde{r}(t,x)$. For $t<T$, it is conceivable that $\PP^\star_t$ is absolutely continuous with respect to $\PP^\circ_t$ because clearly equation \eqref{eq:eta} is solved by 
$ \eta(s,x) =\si(s,x) \left(r(s,x)-\tilde{r}(s,x)\right).$
Contrary to the hypo-elliptic setting in \cite{DelyonHu}, no specific form of the drift needs to be imposed here. However, it is not clear whether
\begin{itemize}
  \item $\|L X^\circ_t - v\|$ tends to zero as $t\uparrow T$;
  \item $\PP^\star_T \ll \PP^\circ_T$. 
\end{itemize}
The two  main results of this paper (Proposition \ref{thm:limitU} and Theorem \ref{thm:main}) provide conditions such that this is indeed the case. Interestingly, in the hypo-elliptic case the necessary ``matching condition'' on the parameters of the auxiliary process $\tilde{X}$ not only involves its diffusion coefficient $\tilde\si(t)$, but its drift $\tilde{b}(t,x)$ as well. In particular, simply equating $\tilde{b}$ to zero turns the measures $\PP^\star_T$ and $\PP^\circ_T$ mutually singular.  For deriving the rate at which $\|L X^\circ_t - v\|$ decays we employ a completely different method of proof compared to the analogous result in \cite{Schauer}, using techniques detailed in \cite{Mao}. While the proof of the absolute continuity result is along the lines of that in \cite{Schauer}, having a partial observation and hypo-ellipticity requires nontrivial adaptations of that proof. 

{\it Put shortly, our results show that guided proposals can be defined for partially observed hypo-elliptic diffusions  exactly as in \cite{Schauer}, if an extra restriction on the drift $\tilde{b}$ of the auxiliary process $\tilde{X}$ is taken into account. } 

Whereas most of the results are derived for $\si$ depending on the state $x$, the  applicability of our methods is mostly confined to the case where $\si$ is only allowed to depend on $t$. The difficulty lies in checking the fourth inequality of Assumption \ref{ass:sigma_const} appearing in Section \ref{sec:main-results}. On the other hand, numerical experiments can give insight whether the law of a particular proposal process and the law of the conditional process are equivalent.

Examples of hypo-elliptic diffusion processes that fall into our setup include
\begin{enumerate}
  \item integrated diffusions, when either the rough, smooth, or both components are observed;
  \item higher order integrated diffusions;
  \item NLCAR models;
  \item the class of hypo-elliptic diffusions considered in \cite{Hairer2011}.
\end{enumerate}
These examples are listed here for illustration purpose. We stress that the derived  results are more general. 

Whereas some examples that we discuss can be treated by the approach of \cite{DelyonHu} (which is restricted to SDEs of the form~\eqref{eq:delyonhu-hypoelliptic}),  our approach extends well beyond this class of models (see for instance Example \ref{ex:not-dh}). Moreover, the hypo-elliptic bridges proposed by \cite{DelyonHu} are bridges of a linear process, whereas the bridges we propose {\it only} use a linear process to derive the guiding term that is superimposed on the true drift. This means that only our approach is able to incorporate  nonlinearity in the drift of the proposal.

\subsection{A toy problem}
Here  we first consider a two-dimensional uniformly elliptic diffusion with unit diffusion coefficient, which is fully observed. Upon taking $\tilde{b}\equiv 0$ and $\tilde\si=\si$, we have
\[ \dd X^\circ_t = b(t,X^\circ_t) \dd t + \frac{v-X^\circ_t}{T-t} \dd t + \dd W_t. \]
The guiding term can be viewed as the distance left to be covered, $v-X^\circ_t$, divided by the remaining time $T-t$. This simple expression is to be contrasted to a hypo-elliptic diffusion, the simplest example perhaps being  an integrated diffusion, with both components observed, i.e.\ a diffusion with 
\[ b(t,x)=\Bm 0 & 1 \\ 0 & 0 \Em x =: B x \quad \text{and} \quad \si=\Bm 0 \\ 1\Em. \] 
It follows from the results  in this paper that using  guided proposals we obtain an ``exact'' proposal, i.e.\  $X_t^{\circ} = X_t^{\star}$ upon taking $\tilde{B}=B$, $\tilde\beta\equiv 0$ and $\tilde\si=\si$. The SDE for $X^\circ$ takes the form 
\[ \dd X^\circ_t = \Bm 0 & 1 \\ 0 & 0 \Em X^\circ_t \dd t + \Bm 0 \\ 18 \frac{v_1-X^\circ_{t,1}}{(T-t)^2} + \frac{10v_2-28 X^\circ_{t,2}}{T-t}\Em \dd t + \Bm 0 \\ 1 \Em \dd W_t, \]
(where $X_{t,i}$ and $v_i$ denote the $i$-th component of $X_t$ and $v$ respectively). This is an elementary consequence of the process being Gaussian and  follows for example directly as a special case of either  Lemma \ref{lem:form-rtilde}  or Equation  \eqref{eq:delyonhu-hypoelliptic}. 

Even for this relatively simple case the guiding term behaves radically different compared to the uniformly elliptic case. 
The pulling term only acts on the rough coordinate and is not any longer inversely proportional to the remaining time.
This illustrates the inherent difficulty of the problem and explains the centring and scaling of $X^\circ$ that we will introduce for studying its behaviour.

\subsection{Outline}
In Section \ref{sec:main-results} we present the main results of the paper. We illustrate the main theorems by applying it to various forms of partially conditioned hypo-elliptic diffusions in Section~\ref{sec:tractable-hypo}. In Section~\ref{sec:numerical} we illustrate our work with simulation examples for the FitzHugh-Nagumo model and a twice integrated diffusion model.  The proof of the proposition on the behaviour of $X^\circ$ near the end-point is given in Section \ref{sec:behaviour-near-endpoint} and  the proof of the theorem on absolute continuity is given in Section \ref{sec:absolute-continuity}.  We end with a discussion in Section \ref{sec:discussion}. Some technical and additional results are gathered in the Appendix. 

\subsection{Frequently used notation}\label{sec:notation}
\subsubsection{Inequalities}
We use the following notation to compare two sequences $\{a_n\}$ and $\{b_n\}$ of positive real numbers: $a_n\lesssim b_n$  (or $b_n\gtrsim a_n$) means that there exists a constant $C>0$ that is independent of $n$ and is such that $a_n\leq C b_n.$ As a combination of the two we write $a_n\asymp b_n$ if both $a_n\lesssim b_n$ and $a_n\gtrsim b_n$. We will also write $a_n \gg b_n$ to indicate that $a_n/b_n\rightarrow\infty$ as $n\rightarrow\infty$. By $a \vee b$ and $a \wedge b$ we denote the maximum and minimum of two numbers $a$ and $b$ respectively. 

\subsubsection{Linear algebra}
We denote the smallest and largest eigenvalue of a square matrix $A$ by $\lmin(A)$ and $\lmax(A)$ respectively. The $p\times p$ identity matrix is denoted by $I_p$. The $p\times q$ matrix with all entries equal to zero is denoted by $0_{p\times q}$. For matrices we  use the spectral norm, which equals the largest singular value of the matrix. The determinant of the matrix $A$ is denoted by $|A|$ and the trace by $\tr(A)$.

\subsubsection{Stochastic processes}
For easy reference, the following table summaries the various processes around. The rightmost three columns give the drift, diffusion coefficient and measure on $C([0,t], \RR^d)$ respectively. 
\begin{center} \rara{1.2}
\begin{tabular}{|l|l|  l | l | l|}
\hline
$X$ & original, unconditioned diffusion process, defined by \eqref{eq:sde} & $b$ & $\si$& $\PP_t$\\
$X^\star$ & corresponding  bridge, conditioned on $v$, defined by \eqref{eq:xstar}& $b^\star$ &$\si$& $\PP^\star_t$\\
$X^\circ$ & proposal process defined by \eqref{eq:prop-sde}& $b^\circ$&$\si$& $\PP^\circ_t$\\ 
$\tilde X$ & linear process defined by  \eqref{eq:auxlin} whose transition densities $\tilde p$ & $\tilde b$ &$\tilde\si$& $\tilde{\PP}_t$\\ &
appear in the definition of $X^\circ$& &&\\
\hline
\end{tabular}
\end{center}
We write 
\[ a(t,x) = \si(t,x)\si(t,x)^\T \quad \text{and}\quad  \tilde{a}(t) = \tilde\si(t) \tilde\si(t)^\T. \]
The state-space of $X$, $X^\star$ and $X^\circ$ is $\RR^d$. The Wiener process lives on $\RR^{d'}$. The observation is determined by the $m\times d$ matrix $L$. Finally, the orthonormal basis $\{f_1,\ldots, f_d\}$ for $\RR^d$ defined in Section \ref{subsubsec:extending-schauer} is fixed throughout, as are the numbers $\xi_1,\ldots, \xi_m$ defined via Equation \eqref{eq:fix-first-mxis}. 

\section{Main results}\label{sec:main-results}

Throughout, we assume
\begin{ass}\label{ass:solutionX-exists}
Both $b$ and $\si$ are globally Lipschitz continuous in both arguments. \end{ass}
This ensures that a strong solution to the SDE \eqref{eq:sde} exists. 
We define the conditioned process, denoted by $X^{\star}$ to be a diffusion process satisfying  the SDE
\begin{equation}\label{eq:xstar}
	\dd X^\star_t = b(t,X^\star_t) \dd t + a(t,X^\star_t) r(t,X^\star_t) \dd t + \si(t,X^\star_t) \dd W_t, \qquad X^\star_0=x_0.
\end{equation} 
Here $r(t,x)=\nabla_x \log \rho(t,x)$. A derivation is given in section \ref{sec:Doob-h}.

\begin{ass}\label{ass:p} The process $X$ has transition densities such that  the mapping $ \rho \colon \RR_+ \times \RR^d \to \RR$ is $C^{\infty,\infty}$ and strictly positive  for all $s< T$ and $x\in \RR^d$. 

  For fixed $x\in \RR^d$, $s$ and $t>s+\eps$,  the mapping $(t,y) \to p(s,x; t,y)$ is continuous and bounded. \end{ass}

In general Assumption~\ref{ass:p} is established by verifying H\"ormander's hypoellipticity conditions; see \cite{Williams1981}. The assumption is satisfied in particular under suitable conditions for the diffusion as described by equations \eqref{eq:descr-hyp1} and \eqref{eq:descr-hyp2}. Note that the results in this paper are not limited to this special case.

\begin{prop}\label{prop:hyp}
Suppose that the matrix valued function $t,x \mapsto \underline{\sigma}$ in  the hypo-elliptic model given by \eqref{eq:descr-hyp1} and \eqref{eq:descr-hyp2} has rank $k'$ for all $(t,x)$. Furthermore suppose that $\underline \sigma$ and $\underline \beta$ are infinitely often differentiable with respect to $(t,x)$. Then the process $(X_t)$ admits a smooth (i.e.\  $C^{\infty}$) density which is also smooth with respect to the initial condition.
\end{prop}
\begin{proof}
This is a special case of Proposition~\ref{prop:hormander} in  Appendix~\ref{sec:hormander}.
\end{proof}

\subsection{Existence of guided proposals}
The guiding term of $X^\circ$ involves $\tilde{r}\colon [0,\infty) \times \RR^d \to \RR$. 
 In the uniformly elliptic case it is easily verified that this mapping is well defined. This need not be the case in the hypo-elliptic setting. 
 
Let $\Phi(t)$ denote the fundamental matrix solution of the ODE 
$\dd \Phi(t) = \tilde{B}(t) \Phi(t) \dd t$, $\Phi(0)=I$ and set $\Phi(t,s)=\Phi(t)\Phi(s)^{-1}$. Define
 \begin{equation}
 	\label{eq:tildeL}
 	L(t) := L \Phi(T,t).
 \end{equation} 
\begin{ass}\label{ass:controll}
	The matrix
\[   \int_t^T \Phi(T,s) \tilde{a}(s) \Phi(T,s)^\T \dd s \] is strictly positive definite for $t<T$.
\end{ass}
In the uniformly elliptic setting, this assumption is always satisfied. 
Under this assumption, the matrix
	\[ \Md(t) := \int_t^T L(s) \tilde{a}(s) L(s)^\T \dd s \]
is also strictly positive definite for all $t\in [0,T)$ and, in particular, invertible.

\begin{lem}\label{lem:form-rtilde}
If Assumption \ref{ass:controll} holds then,
\begin{equation}
	\label{eq:tilder}
	\tilde{r}(t,x)=
L(t)^\T M(t) \left( v -\mu(t)-L(t)x\right), \qquad t\in [0,T],
\end{equation} 
where\[ \mu(t)=\int_t^T L(s) \tilde\beta(s) \dd s\] and
\begin{equation} \label{eq:tildeM}M(t) = [\Md(t)]^{-1}.\end{equation}

\end{lem}
\begin{proof}
The solution to the SDE for $\tilde{X}_u$ is given by 
\[ \tilde{X}_u = \Phi(u,t) x + \int_t^u \Phi(u,s) \tilde\beta(s) \dd s + \int_t^u \Phi(u,s) \tilde{\sigma}(s) \dd W_s, \quad u\ge t ,\quad \tilde{X}_t =x. \]
Cf.\ \cite{LiptserShiryayevI}, Theorem 4.10. The result now follows directly upon taking $u=T$, multiplying both sides with $L$ and using the definition of $L(t)$. 
\end{proof}

In Appendix \ref{sec:app-control} easily verifiable conditions for the existence of $\tilde{p}$ are given for the case $L=I$.

Since $t\mapsto \mu(t)$ and $t\mapsto M(t)$ are continuous and $x\mapsto \tilde{r}(t,x)$ is linear in $x$ for fixed $t$, the process $X^\circ$ is well defined on intervals bounded away from $T$. 
\begin{lem}
Under Assumptions \ref{ass:solutionX-exists} and \ref{ass:controll} we have that 
for any $t<T$, the SDE for $X^\circ$ has a unique strong solution on $[0,t]$. 
\end{lem}

{\it Throughout, without explicitly stating it in lemmas and theorems, we will assume that Assumptions \ref{ass:solutionX-exists}, \ref{ass:p} and  \ref{ass:controll} hold true.}

\subsection{Behaviour of guided proposals near the endpoint}\label{sec:behav}

Let $\Delta(t)$ be an invertible $m\times m$ diagonal matrix-valued measurable function on $[0,T)$. Define
\begin{equation}\label{eq:defU} \ZD{t}=\Delta(t) \left( v- \mu(t)-L(t) X^\circ_t  \right) \end{equation} and
\begin{equation}\label{eq:def-tildeC-P}
	\LD(t)=\Delta(t) L(t) \qquad \MD(t)=\Delta(t)^{-1} M(t) \Delta(t)^{-1}.
\end{equation} 
 Note that for $t\approx T$,  we have $\Phi(T,t)\approx I$ and hence $\ZD{t} \approx \Delta(t)(v-L X^\circ_t)$. The matrix $\Delta(t)$ is a scaling matrix which in the hypo-elliptic case incorporates the difference in rate of convergence for smooth and rough components of $L X^\circ_t$ to $v$, when $t\uparrow T$. In the uniformly elliptic case, we can always take $\Delta(t)=I_{m}$. 

The following assumption is of key importance.
\begin{ass}\label{ass:sigma_const}
There exists an invertible $m\times m$ diagonal matrix-valued function  $\Delta(t)$, which is measurable on $[0,T)$, a $t_0<T$, $\alpha \in (0,1]$ and positive constants $\underline{c}$, $\overline{c}$, $c_1$,   $c_2$ and $c_3$ such that for all $t \in [t_0,T)$
\begin{equation}\label{eq:bounds-thm}
\begin{split}
\underbar{c}\, (T-t)^{-1}\le \lmin\left(\MD(t)\right)  &\le \lmax\left(\MD(t)\right) \le \overline{c}\, (T-t)^{-1}, \\
\left\|  \LD(t)  \left(\tilde{b}(t,x) - b(t,x)\right) \right\| &\le c_1 \\
 \tr\left(\LD(t)\, a(t,x)\, \LD(t)^\T \right)
	&\le c_2  \\
	\|\LD(t) (\tilde{a}(t)-a(t,x)) \LD(t)^\T\| &\le c_3 (T-t)^\alpha.
	\end{split}	
\end{equation}

\end{ass}

\begin{prop}\label{thm:limitU}
Under Assumption \ref{ass:sigma_const}, there exists a positive number $C$ such that
\[ \limsup_{t\uparrow T} \frac{\|\ZD{t}\|}{\sqrt{(T-t) \log(1/(T-t))}} \le C \qquad \text{a.s.} \]
\end{prop}

\begin{rem}
If $\si$ is state-dependent, it is particularly difficult to ensure that the fourth inequality in \eqref{eq:bounds-thm} is satisfied. There is at least one non-trivial example where this inequality can be assured (see Example \ref{ex:state-dep}). In Section \ref{sec:discussion} we further discuss the case of state dependent diffusivity.    In the simpler  case where $\si$ only depends on $t$, we can always take $\tilde\si(t)=\si(t)$ and then the fourth inequality is trivially satisfied. In Section \ref{sec:tractable-hypo} we verify \eqref{eq:bounds-thm}  for a wide range of examples. As a prelude: for the SDE system specified by \eqref{eq:descr-hyp1} and \eqref{eq:descr-hyp2} one takes $\tilde{B}=B$ and $\tilde\si=\si$. Then $\Delta(t)$ can be chosen such that the first inequality is satisfied. The second condition of \eqref{eq:bounds-thm} encapsulates a matching condition on the drift which induces some restrictions on $\tilde\beta$ and $\underline\beta$. The third inequality is then usually satisfied automatically. 
\end{rem}

The uniformly elliptic case is particularly simple:
\begin{cor}[Uniformly elliptic case]\label{cor:limit-behav-ue}
Assume that either (i) the diffusivity    $\si$ is constant and  $\tilde\si=\si$ or (ii) $\si$ depends on $t$ and $\tilde\si(t)=\si(t)$. Assume  $a$ is strictly positive definite and that 
 $b(t,x)-\tilde{b}(t,x)$ is bounded on $[0,T]\times \RR^d$. Then the conclusion of Proposition \ref{thm:limitU} holds true with $\Delta(t)=I_m$.
\end{cor}
\begin{rem}
The behaviour of $\|\ZD{t}\|$ that we obtain agrees with the results of \cite{Schauer}. That paper is confined to $L=I$ and the uniformly elliptic case, but includes the case of state-dependent diffusion coefficient. 
Under this assumption, it suffices that  $\tilde\sigma(T)^\T \tilde\sigma(T) = a(T,v)$, a condition that can always be ensured to be satisfied.

	\end{rem}

The proofs of Theorem \ref{thm:limitU} and Corollary \ref{cor:limit-behav-ue} are given in Section \ref{sec:behaviour-near-endpoint}.

In Section \ref{sec:tractable-hypo} we give a set of tractable hypo-elliptic models for which the conclusion of Theorem \ref{thm:limitU} is valid. The appropriate choice of the scaling matrix $\Delta(t)$ is really problem specific. Moreover, the assumptions on the auxiliary process depend on the choice of $L$. 

In most cases it will not be possible to satisfy the fourth inequality of Assumption \ref{ass:sigma_const} when the diffusion coefficient is state-dependent. The following example shows an exception. 
\begin{ex}\label{ex:state-dep}
Suppose the diffusion is uniformly elliptic and $L=\Bm \underline{L} & 0_{m\times k'}\Em$, where $L \in \RR^{m\times k}$ and $d=k+k'$.  Now suppose $a(t,x)$ is of block form:
\[ a(t,x) =\Bm a_{11}(t) & 0_{k\times k'} \\ 0_{k'\times k} & a_{22}(t,x) \Em \]
and that  we take $\tilde{a}$ to be of the same block form. 
Upon taking $\tilde{B} =0_{d\times d}$ and $\Delta(t)=I_d$, we see that $\LD(t)=L$ and hence 
\[	\LD(t) (\tilde{a}(t)-a(t,x)) \LD(t)^\T = \underline{L}  (\tilde{a}_{11}(t)-a_{11}(t)) \underline{L}^\T. \]
Therefore, if we choose $\tilde{a}_{11}(t)$ to be equal to $a_{11}(t)$ the fourth inequality in Assumption \ref{ass:sigma_const} is trivially satisfied. 
\end{ex}
Empirically however, it appears that Assumption \ref{ass:sigma_const} is stronger than needed for valid guided proposals, see Example \ref{example:sigma(x)}.

\subsection{Absolute continuity}

The following theorem gives sufficient conditions for absolute continuity of $\PP_T^\star$ with respect to $\PP_T^\circ$. First we introduce an assumption.

\begin{ass}\label{ass:p<=tildep}
There exists a constant $C$ such that 
\begin{equation}\label{eq:p<=tildep} p(s,x;t,y) \le C \tilde{p}(s,x; t,y) \qquad 0\le s<t<T \end{equation}
for all $x, y \in \RR^d$. 
\end{ass}

\begin{thm}\label{thm:main}
Assume there exists a positive $\delta$ such that $|\Delta(t)| \lesssim (T-t)^{-\delta}$.
If Assumptions \ref{ass:sigma_const} and \ref{ass:p<=tildep} hold true, then\[ \frac{\dd \PP_T^\star}{\dd \PP_T^\circ}(X^\circ) = \frac{\tilde \rho(0,x_0)}{ \rho(0,x_0)}\Psi_T(X^\circ), \]
where 
\begin{equation}\label{eq:Psi} \Psi_t(X^\circ)=\exp\left(\int_0^t \scr{G}(s,X^\circ_s) \dd s\right), \end{equation}
\[
\begin{split}\label{eq:G} \scr{G}(s,x) &= (b(s,x) - \tilde b(s,x))^\T \tilde r(s,x)  \\ & \qquad -  \frac12 \tr\left(\left[a(s,x) - \tilde a(s)\right] \left[\tilde H(s)-\tilde{r}(s,x)\tilde{r}(s,x)^\T\right]\right)
\end{split}
\]
and $\tilde{H}(s)=L(s)^\T M(s) L(s)$.
\end{thm}

The proof is given in Section \ref{sec:absolute-continuity}. 
\begin{rem}
	The expression for the Radon-Nikodym derivative does depend on the intractable transition densities $p$ via the term $\rho(0,x_0)$. This is a multiplicative term that only shows up in the denominator and therefore cancels in the acceptance probability for sampling diffusion bridges using the Metropolis-Hastings algorithm. 
	
\end{rem}

The following lemma is useful for verifying Assumption \ref{ass:p<=tildep}. Its proof is located in Section \ref{sec:absolute-continuity}. 
\begin{lem} \label{lem:verif-p<=tildep}
Assume $\eta(s,x)$ satisfies the equation
\[ \si(s,x) \eta(s,x) = b(s,x)-\tilde{b}(s,x) \]
and that $\eta$ is bounded. Then there exists a constant $C$ such that 
\[ p(s,x;t,y) \le C \tilde{p}(s,x; t,y)\]
for all $x, y \in \RR^d$  for all  $0\le s<t\le T $. 
\end{lem}


\section{Tractable hypo-elliptic models} \label{sec:tractable-hypo}

In this section we give several examples of hypo-elliptic models that satisfy Assumption \ref{ass:sigma_const}. 
In the following we write $X_t=\Bm X_{t,1} & \cdots & X_{t,d}\Em^\T$. 

In each of the examples  we choose  an appropriate scaling matrix  $\Delta(t)$  and verify the conditions of Assumption \ref{ass:sigma_const}. For this,  we need to evaluate $\LD(t)$ and $\MD(t)$. The computations are somewhat tedious by hand (though straightforward), and for that reason we used the computer algebra system {\sf Maple}  for this. Ideally, instead of the conditions appearing  in Assumption \ref{ass:sigma_const}, one would like to have conditions only containing  $b$, $\tilde b$, $\si$ and $\tilde \si$. This however seems hard to obtain and maybe a bit too much to ask for, given the wide diversity in  behaviour of hypo-elliptic diffusions and the generality of the matrix $L$. In each of the examples, we state the model and the conditions on $\tilde{b}$ and $\tilde\si$ such that Assumption \ref{ass:sigma_const} is satisfied.

\begin{ex}\label{ex:integrated-diffusion} {\bf Integrated diffusion, fully observed.}
Suppose 
\begin{equation}\label{eq:intdiffusion}
\begin{split}
	\dd X_{t,1} & = X_{t,2} \dd t \\
	\dd X_{t,2} & = \underline\beta(t,X_t) \dd t + \ga \dd W_t,	\end{split}	
\end{equation}
where $\underline\beta\colon [0,T] \times \RR^2 \to \RR$ is bounded and globally Lipschitz in both arguments. If $L=I_2$, and the coefficients of the auxiliary process $\tilde{X}$ satisfy
\[ \tilde{B}(t) = \Bm 0 & 1 \\ 0 & 0 \Em, \qquad \tilde\beta_1(t) = 0, \qquad \tilde\sigma =\Bm 0 \\ \gamma\Em \]
then Assumption \ref{ass:sigma_const} is  satisfied.

{\it Proof:} 
 As we expect the rate of the first component, which is smooth, to converge to the endpoint one order faster than the second component, which is rough, we take
\[ \Delta(t)=\Bm (T-t)^{-1} & 0 \\ 0 & 1 \Em. \]
We have 
\[ b(t,x)-\tilde{b}(t,x) =  \Bm 0 \\ \underline\beta(t,x)\Em. \]
By choice of $\tilde{B}$ and $\Delta$ we get
\[  \LD(t)=\Delta(t) L \Phi(T,t) =\Delta(t) \Phi(T,t)= \Bm 1/(T-t) & 1 \\ 0 & 1 \Em. \]
and
\[ M(t) = \frac1{\ga^2} \Bm 12/(T-t)^3 & 6/(T-t)^2 \\6/(T-t)^2 & 4/(T-t) \Em \quad \Longrightarrow\quad  \MD(t) =  \frac1{\ga^2} (T-t)^{-1} \Bm 12 & 6 \\ 6 & 4\Em. \]
Now it is trivial to verify that Assumption \ref{ass:sigma_const} is  satisfied.
\end{ex}

 \begin{ex}\label{ex:integrated-diffusion-partial1} {\bf Integrated diffusion, smooth component observed.}
 Consider the same setting as in the previous example, but now with $L=\Bm 1 & 0\Em$. That is,  only the smooth (integrated) component is observed. Then Assumption \ref{ass:sigma_const} is satisfied.

 {\it Proof:}
 Upon taking  $\Delta(t)=(T-t)^{-1}$ we get \begin{equation}\label{eq:intdiff:smooth-obs}  \MD(t)=3\ga^{-2}(T-t)^{-1} \quad \text{and} \quad \LD(t)=\Bm 1/(T-t) & 1 \Em\end{equation}
from which the claim easily follows
\end{ex}

 \begin{ex}\label{ex:integrated-diffusion-partial2} {\bf Integrated diffusion, rough component observed.}
 Consider the same setting as in Example \ref{ex:integrated-diffusion}, but now with $L=\Bm 0 & 1\Em$. That is,  only the rough component is observed. Then Assumption \ref{ass:sigma_const} is satisfied.

{\it Proof:}
Taking $\Delta(t)=1$ we get
 $ \MD(t)=\ga^{-2}(T-t)^{-1}$ and $\LD(t)=\Bm 0 & 1 \Em,$
 from which the claim easily follows.

The guiding term is completely independent of the first component. This is not surprising, as this example is equivalent to fully observing a one-dimensional uniformly elliptic diffusion (described by the second component). 
\end{ex}

\begin{ex}\label{ex:NCLAR} {\bf NLCAR($p$)-model.} The integrated diffusion model is a special case of the class of continuous-time nonlinear autoregressive models (Cf.\ \cite{StramerRobertsJTSA}). The real valued process $Y$ is called a $p$-th order NLCAR model if it solves the $p$-th order SDE
\[ \dd Y^{(p-1)}_t = \underline\beta(t,Y_t) \dd t + \ga \dd W_t. \]
We assume  $\underline\beta\colon [0,T] \times \RR^2 \to \RR$ is bounded and globally Lipschitz in both arguments.
This example corresponds to the model specified by \eqref{eq:descr-hyp1}--\eqref{eq:descr-hyp2} with $d=p$, $d'=1$ and $k=p-1$. 	Integrated diffusions correspond to $p=2$. Observing only the smoothest component means that we have $L=\Bm 1 & 0_{1\times d-1} \Em$. This class of models includes in particular continuous-time autoregressive  and continuous-time threshold autoregressive models, as defined in \cite{Brockwell-carma}.

We consider the NCLAR($3$)-model more specifically, which can be written explicitly as a diffusion in $\RR^3$ with 
\begin{equation}\label{eq:NLCAR3} b(t,x) = \Bm 0 & 1 & 0 \\ 0 & 0 & 1 \\ 0 & 0 & 0 \Em x + \Bm 0 \\ 0 \\ \underline{\beta}(t,x)\Em \qquad \si=\Bm 0 \\ 0 \\ \ga\Em. \end{equation}

If either $L=I_3$ or  $L=\Bm 1 & 0 & 0\Em$, Assumption \ref{ass:sigma_const} is satisfied if the coefficients of the auxiliary process $\tilde{X}$ satisfy
\begin{equation}\label{eq:aux-nclar3} \tilde{B}(t) = \Bm 0 & 1&0 \\ 0 & 0&1 \\ 0 & 0 &0 \Em, \qquad \tilde\beta_1(t)=\tilde\beta_2(t) = 0, \qquad \tilde\sigma =\Bm 0 \\0\\ \gamma\Em. \end{equation}

{\it Proof:}
If $L=I_3$ we   take the scaling matrix
\[ \Delta(t)=\Bm (T-t)^{-2} & 0&0 \\ 0 & (T-t)^{-1}&0\\ 0 & 0 & 1 \Em \]
to account for the different degrees of smoothness of the paths of the diffusion. We then obtain, defining $w(t)=(T-t)^{-1}$
\[ \MD(t)= \frac{3w(t)}{\ga^2}\Bm 240 & -120 & 20 \\
-120 & 64 & -12 \\ 20 & -12 & 3\Em\quad  \text{and} \quad \LD(t)=\Bm w(t)^2 & w(t) & 1/2 \\ 0 & w(t) & 1 \\ 0 & 0 &1\Em, \]
from which the claim is easily verified. 

In case $L=\Bm 1 & 0 & 0\Em$, we take $\Delta(t)=(T-t)^{-2}$  and since
\[ \MD(t) = \frac{20}{\ga^2} (T-t)^{-1} \quad \text{and}\quad \LD(t)=\Bm (T-t)^{-2} & (T-t)^{-1} & 1/2\Em. \]
Assumption \ref{ass:sigma_const} is satisfied. 
See Example~\ref{numex:nclar3} for a numerical illustration of this example.
\end{ex}

\begin{ex} \label{ex:FM-demodulation}
Assume the following model for FM-demodulation:
\[
	\dd \Bm X_{t,1} \\ X_{t,2}\\ X_{t,2}\Em  = 
	\Bm X_{t,2} \\ -\alpha X_{t,2} \\ \sqrt{2\ga} \sin(\omega t + X_{t,1}) \Em \dd t + \Bm 0 & 0 \\ \sqrt{2\ga \alpha} & 0 \\ 0 & \psi \Em \dd \Bm W_{t,1} \\ W_{t,2} \Em. 
\]
Here, the observation is determined by $L=\Bm 0 & 0 & 1\Em$. 
Motivated by this example, we check our results for a diffusion with coefficients
\[ b(t,x) = B x + \Bm 0 \\ \underline\beta_2(t,x) \\ \underline\beta_3(t,x) \Em,	\qquad B=\Bm 0 & 1 & 0 \\ 0 & -\alpha & 0 \\0 & 0 & 0 \Em
\quad\text{and}\quad 
\si = \Bm 0 & 0 \\ \ga_1 & \ga_2 \\ \ga_3 & \ga_4\Em. \]
Note that this is a slight generalisation of the FM-demodulation model. We will assume that $\ga_3^2+\ga_4^2\neq 0$ and  $\underline\beta_3$ to be  bounded. 
If  $\tilde{B}(t)=B$ and $\tilde\si=\sigma$, then Assumption \ref{ass:sigma_const} is satisfied.
 
{\it Proof:}
As the observation is  on the rough component, we choose  $\Delta(t)=1$. We have \[ \MD(t)=(T-t)^{-1}(\ga_3^2+\ga_4^2)^{-1}\]
and $\LD(t)=\Bm 0 & 0 & 1\Em$. 
Hence 
$  \LD(t)\left(\tilde{b}(t)-b(t,x)\right) = \tilde\beta_3(t)-\underline\beta_3(t,x) $
and the other conditions are easily verified.
\end{ex}

\begin{ex}\label{ex:tracking-2D}
Assume $\Bm X_{t,1} & X_{t,2}\Em^\T$ gives the position in the plane of a particle at time $t$. Suppose the velocity vector of the particle at time $t$, denoted by $\Bm X_{t,3} & X_{t,4}\Em^\T$ satisfies a SDE driven by a $2$-dimensional Wiener process. The evolution of $X_t=\Bm  X_{t,1} & X_{t,2} & X_{t,3} & X_{t,4}\Em^\T$ is then described by the SDE
\begin{align*}
	\dd X_{t,1} & = X_{t,3} \dd t \\
	\dd X_{t,2} & = X_{t,4} \dd t \\
	\dd \Bm X_{t,3}\\ X_{t,4}\Em & = \Bm \underline\beta_3(t,X_t) \\ \underline\beta_4(t,X_t)\Em \dd t + \ga \dd W_t,
\end{align*}
where $W_t \in \RR^2$. 
This example corresponds to the case $d=4$, $d'=2$ and $k=2$ in the model specified by \eqref{eq:descr-hyp1}--\eqref{eq:descr-hyp2}. Observing only the location corresponds to
\[ L=\Bm 1 & 0 & 0&0 \\
0 & 1 & 0 & 0 \Em. \]

In matrix-vector notation the drift of the diffusion is given by $b(t,x) = B x + \beta(t,x)$, where 
\[ B=\Bm 0 & 0 & 1 &0\\ 0 & 0 & 0&1 \\0 & 0 & 0 &0\\0 & 0 & 0 &0\Em \qquad\text{and}\qquad  \beta(t,x) =\Bm 0 \\ 0 \\ \underline\beta_3(t,X_t)\\\underline\beta_4(t,X_t)\Em. \]
We will assume diffusion coefficient
$\displaystyle	 \si=\Bm 0 & 0 & \ga_1 & \ga_3\\ 0&0&\ga_2 & \ga_4\Em^\T, $
where $\ga_1\ga_4-\ga_2\ga_3 \neq 0$. 
If $\tilde\beta_1(t)=\tilde\beta_2(t)=0$, $\tilde{B}(t)=B$ and $\tilde\si=\si$, then Assumption \ref{ass:sigma_const} is satisfied.

{\it Proof:} 
As we observed the first two coordinates, which are both smooth, we take $\Delta(t)=(T-t)^{-1}I_2$. The claim now follows from \[ \MD(t)=(T-t)^{-1}\frac{3}{(\ga_1\ga_4-\ga_2\ga_3)^2}\Bm 
-\ga_3^2-\ga_4^2 & \ga_1\ga_3+\ga_2\ga_4\\ \ga_1\ga_3+\ga_2\ga_4& -\ga_1^2-\ga_2^2.
\Em\]
and 
$\displaystyle \LD(t) = \Bm (T-t)^{-1} & 0 & 1 & 0 \\ 0 & (T-t)^{-1} & 0 & 1 \Em.$
\end{ex}

\begin{ex}\label{ex:hairer-stuart-voss}
\cite{Hairer2011} consider SDEs of the form
\[ \dd X_t = \Bm 0 & 1 \\ 0 & \th\Em X_t\dd t + \Bm 0 \\ \underline\beta(t,X_t) \Em \dd t + \Bm 0 \\ \gamma \Em \dd W_t,
\] where $X_t=\Bm X_{t,1} & X_{t,2}\Em^\T$, $\th >0$ and the conditioning is specified by $L=\Bm 1 & 0\Em$. As explained in \cite{Hairer2011}  the solution to this SDE can be viewed as the time evolution of the state of a mechanical system with friction under the influence of noise. Assume $(t,x) \mapsto \underline\beta(t,x)$ is bounded and Lipschitz in both arguments. Note that this hypo-elliptic SDE is not of the form given in \eqref{eq:descr-hyp1} and \eqref{eq:descr-hyp2}. 
However, if  \[ \tilde{B}(t)= B_\theta= \Bm 0 & 1\\ 0 & \th\Em, \quad \tilde\beta_1(t)=0\quad  \text{and} \quad \tilde\si=\Bm 0 \\ \gamma \Em, \] 
then Assumption \ref{ass:sigma_const} is satisfied.

{\it Proof:} 
Upon 	taking $\Delta(t)=(T-t)^{-1}$, we find that 
\[ \lim_{t\uparrow T} (T-t) \MD(t) = 3\ga^{-2} \quad\text{and} \quad  \LD(t) =\Bm \frac1{T-t} & \frac{\e^{\th (T-t)}-1}{\th (T-t)}\Em. \]
This is to be compared with the expressions in \eqref{eq:intdiff:smooth-obs}. We conclude similarly as in  Example \ref{ex:integrated-diffusion-partial1}  that the conditions in Assumption \ref{ass:sigma_const} are satisfied.

\end{ex}

\begin{ex}\label{ex:not-dh}
This is an example to illustrate that our approach applies beyond equations of the form \eqref{eq:delyonhu-hypoelliptic}. We assume 
\[ \dd X_t =  B X_t\dd t + \Bm 0 \\ \underline\beta(t,X_t) \Em \dd t + \Bm 1 \\ 1 \Em \dd W_t
\]
with $\underline\beta\colon [0,T] \times \RR^2 \to \RR$ bounded and globally Lipschitz in both arguments.  
Suppose $L=\Bm 1 & 0\Em$. If 
\[ \tilde{B} =B:= \Bm 0 & 1 \\ 0 & 0\Em, \qquad \tilde\beta=\Bm 0 \\ 0 \Em, \qquad \tilde\si = \Bm 1\\ 1 \Em, \]
then \ref{ass:sigma_const}  holds. 

{\it Proof:}
Using $\Delta(t)=1$,  we have $L_\Delta(t)=\Bm 1 & T-t\Em$, $\lim_{t\uparrow T} (T-t) \MD(t)=1$ and the claim follows as in the previous examples. 

\end{ex}

\section{Numerical illustrations}\label{sec:numerical}

In this section we will discuss implementational aspects of our sampling method, and we will illustrate the method by some representative numerical examples. 
We implemented the examples as parts of the authors' software package \textsf{Bridge} (\cite{bridge}), written in the programming language
\textsf{Julia}, (\cite{julia}.) The corresponding code is available, \cite{codeexamples}.

For computing the guiding term and likelihood ratio, we have the following backwards ordinary differential equations
\begin{alignat}{3}
 \dd L(t) &= -L(t) \tilde{B}(t)\dd t,&L(T)&= L  \\
 \dd \Md(t)&=- L(t) \tilde{a}(t) L(t)^\T\dd t,\quad& \Md(T)&=0_{m\times m}  \\
 \dd \mu(t) &=-L(t) \tilde\beta(t)\dd t,& \mu(T)&=0_{m\times 1},  
\end{alignat}
where $t\in [0,T]$. These are easily derived, cf.\  lemma 2.4 in \cite{Vdm-S-smoothing}. These backward differential equations need only be solved once. Next, Algorithm 1 from \cite{vdm-s-estpaper} can be applied. This algorithm describes  a Metropolis-Hastings sampler for simulating diffusion bridges using guided proposals. We briefly recap the steps of this algorithm, more details can be found in \cite{vdm-s-estpaper}. As we assume $X^\circ$ to be a strong solution to the SDE specified by Equations \eqref{eq:prop-sde} and \eqref{eq:guideddrift}, there is a measurable mapping $\scr{G}\scr{P}$  such that $X^\circ=\scr{G}\scr{P}(x_0, W)$, where $x_0$ is the starting point and $W$ a $\RR^{d'}$-valued Wiener process ($\scr{G}\scr{P}$ abbreviates $\scr{G}$uided $\scr{P}$roposal). As $x_0$ is fixed, we will write, with slight abuse of notation, $X^\circ=\scr{G}\scr{P}(W)$. The algorithm requires to choose a tuning parameter $\rho\in [0,1)$ and proceeds according to the following steps.
\begin{enumerate}
  \item Draw a Wiener process $Z$ on $[0,T]$, Set $X=g(Z)$.
  \item Propose a Wiener process $W$ on $[0,T]$. Set 
  \begin{equation}\label{eq:langevinprop} Z^\circ = \rho Z + \sqrt{1-\rho^2} W \end{equation}  and 
  $X^\circ=\scr{G}\scr{P}(Z^\circ)$.
  \item Compute $A=\Psi_T(X^\circ)/\Psi_T(X)$ (where $\Psi_T$ is as defined in  \eqref{eq:Psi}). Sample $U\sim \mbox{Uniform}(0,1)$. If $U<A$ then set $Z=Z^\circ$ and $X=X^\circ$.
  \item Repeat steps (2) and (3).
\end{enumerate}
The invariant distribution of this chain is precisely $\PP^\star_T$. 
If the guided proposal is good, then we may use $\rho=0$, which yields an independence sampler. However, for difficult bridge simulation schemes, possibly caused by a large value of $T$ or strong nonlinearity in the drift or diffusion coefficient, a value of $\rho$ close to $1$ may be required. The proposal in Equation \eqref{eq:langevinprop} is precisely the pCN method, see e.g.\ \cite{cotter2013}. 

In the implementation we use a fine equidistant grid, which is transformed by the mapping $\tau\colon [0,T] \to [0,T]$ given by $\tau(s)=s(2-s/T)$. Motivation for this choice is given in Section 5 of \cite{vdm-s-estpaper}. Intuitively, the guiding term gets stronger near $T$ and therefore we use a finer grid the closer we get to $T$.  The guided proposal is simulated on this grid, and using the values obtained, $\Psi_T(X^\circ)$ is approximated by Riemann approximation. Furthermore, for numerical stability we solve the equation for $\Md(t)$ using $\Md(T)= 10^{-10} I_{m\times m}$ instead of  $\Md(T)=0_{m\times m}$.

\begin{ex}\label{numex:nclar3}
Assume the NCLAR($3$)-model, as described in Example \ref{ex:NCLAR} with $\underline\beta(t,x) =-6\sin(2\pi x)$ and $x_0=\Bm 0 & 0 & 0\Em^\T$. 
We first condition the process on hitting \[v=\Bm 1/32 & 1/4 & 1\Em^\T\] at time $T=0.5$, assuming $L=I_3$ (full observation at time $T$). The idea of this example is that sample paths of the rough component are mean-reverting at levels $k \in \ZZ$, with occasional noise-driven shifts from one level to another. The given conditioning then  forces the process to move halfway the interval (at about time $0.25$) from level $0$ to level $1$, remaining at level approximately level $1$ up till time $T$. Such paths are rare events and obtaining these by forward simulation is computationally extremely intensive.
 
We construct guided proposals according to \eqref{eq:aux-nclar3} with $\tilde\beta_3(t)=0$. 
Iterates of the sampler using $\rho=0.85$ are shown in Figure \ref{fig:nclar-fullcond}. The average Metropolis-Hastings acceptance percentage was $43\%$. We need a value of $\rho$ close to $1$ as we cannot easily incorporate the strong nonlinearity into the guiding term of the guided proposal.  We repeated the simulation, this time only conditioning on $LX_T=1/32$, where $L=\Bm 1 & 0 & 0\Em$. We again took $\rho=0.95$, leading to an average Metropolis-Hastings acceptance percentage of $24\%$. The results are in Figure \ref{fig:nclar-partialcond}. The distribution of bridges turns out to be bimodal. The latter is confirmed by extensive forward simulation and only keeping those paths which approximately satisfy the conditioning.

\begin{figure}
	\begin{center}
			\includegraphics[scale=0.7]{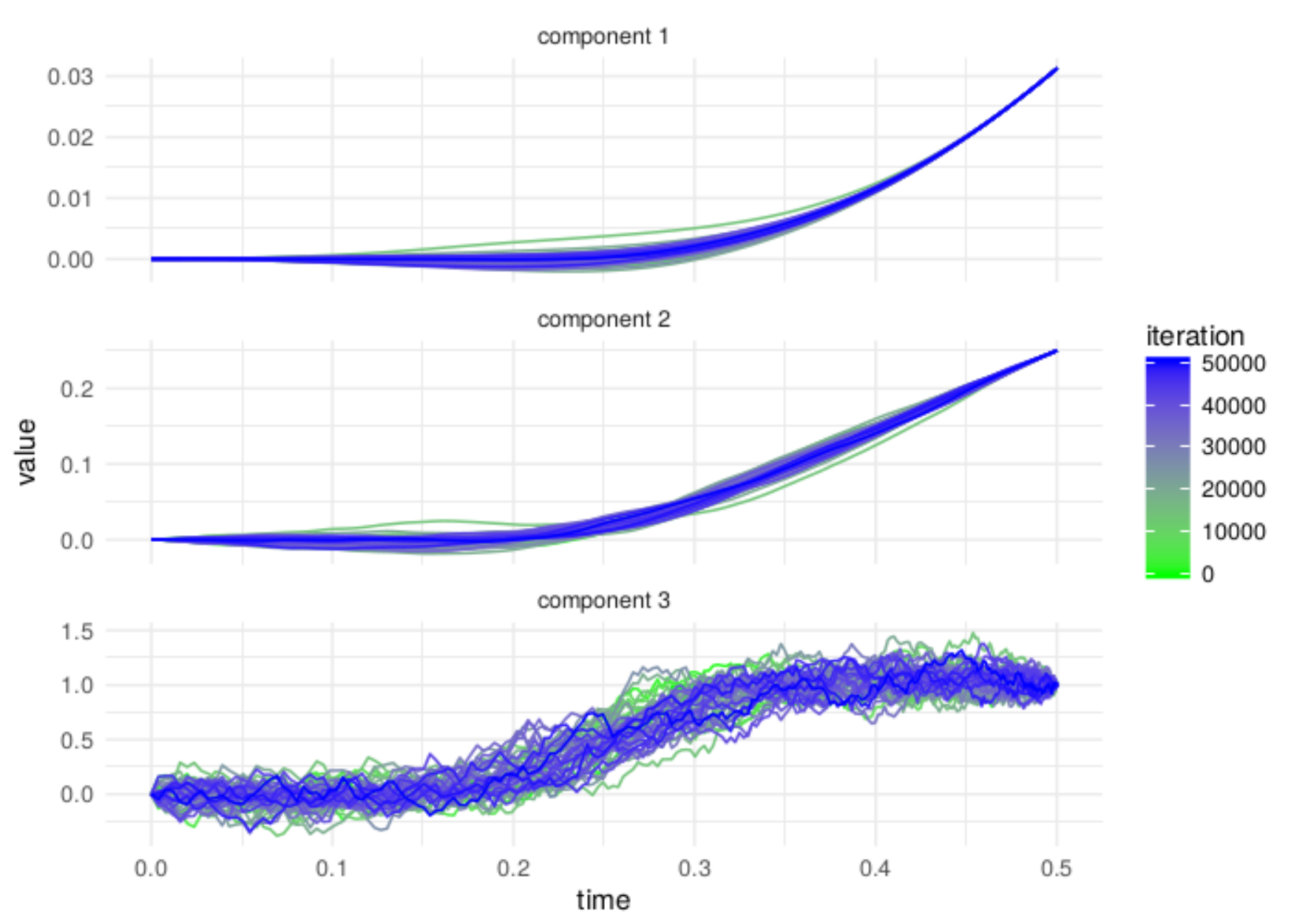}
	\end{center}
	\caption[]{Sampled guided diffusion bridges when conditioning on 
$X_T=\Bm 1/32 & 1/4 & 1\Em^\T$ 	 in Example \ref{numex:nclar3}.	\label{fig:nclar-fullcond}}

\end{figure}

\begin{figure}
	\begin{center}
			\includegraphics[scale=0.7]{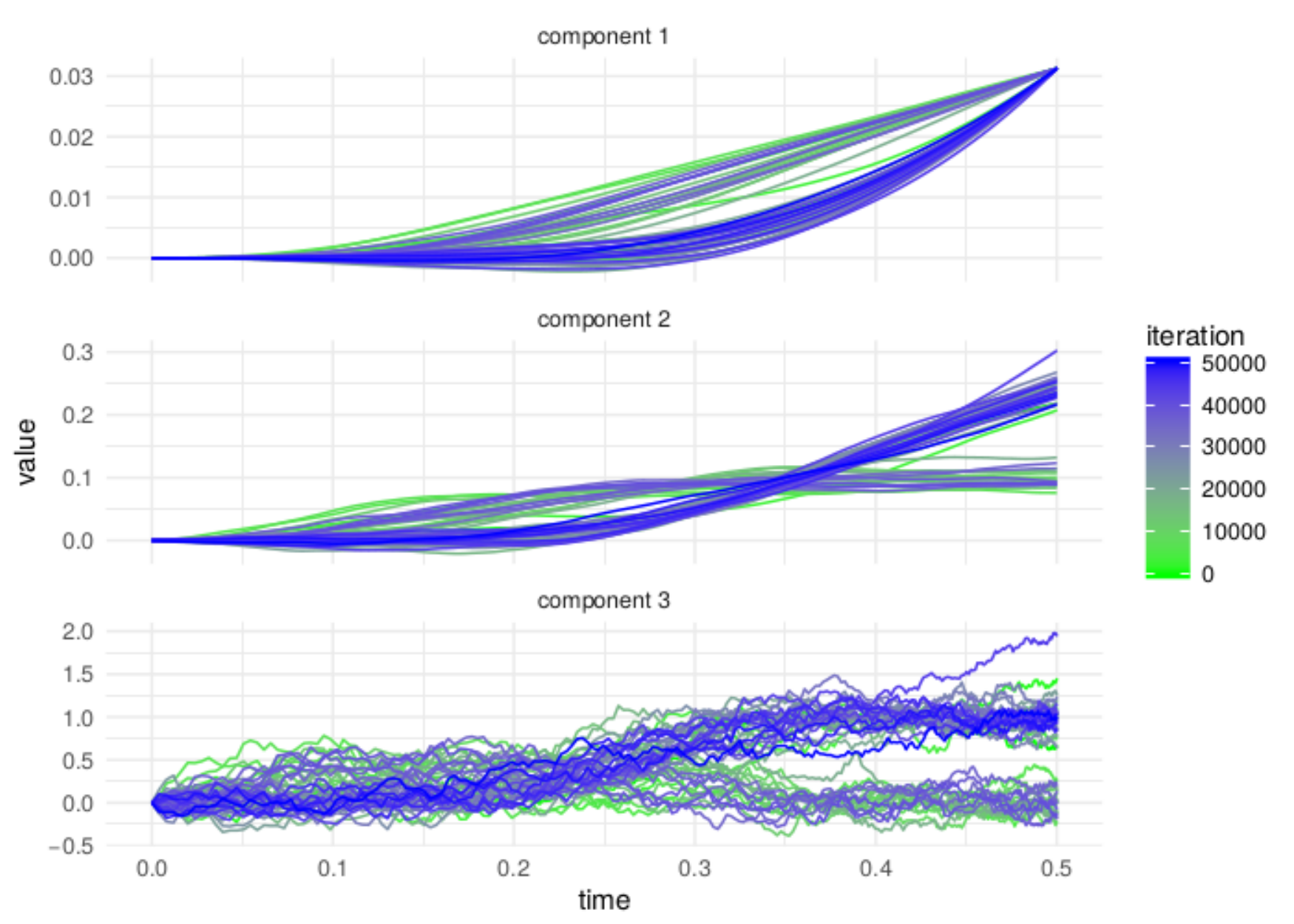}
	\end{center}
	\caption[]{Sampled guided diffusion bridges when conditioning on  $L X_T=1/32$ with 
		$L=\Bm 1 & 0 & 0\Em^{\T}$ 
		 in Example \ref{numex:nclar3}.}
	\label{fig:nclar-partialcond}
\end{figure}

\end{ex}

\begin{ex}
\cite{ditlevsen-samson2017} consider the stochastic hypo-elliptic FitzHugh-Nagumo model, which is specified by the SDE
\begin{equation}\label{eq:fitz} \dd X_t = \Bm 1/\eps & -1/\eps \\ \ga & -1 \Em X_t \dd t + \Bm -X_{t,1}^3/\eps +s/\eps \\ \beta\Em \dd t + \Bm 0 \\ \si\Em \dd W_t,\qquad X_0=x(0). \end{equation}
Only the first component is observed, hence $L=\Bm 1 & 0\Em$. We consider the same parameter values as in \cite{ditlevsen-samson2017}:
\begin{equation}\label{eq:fitz-parvalues}
 \Bm \eps & s & \ga & \beta & \sigma\Em =\Bm 0.1& 0 & 1.5 & 0.8 & 0.3\Em. \end{equation}
A realisation of a sample path on $[0,10]$ is given in Figure \ref{fig:fh-samplepath}. 
\begin{figure}
	\begin{center}
			\includegraphics[scale=0.7]{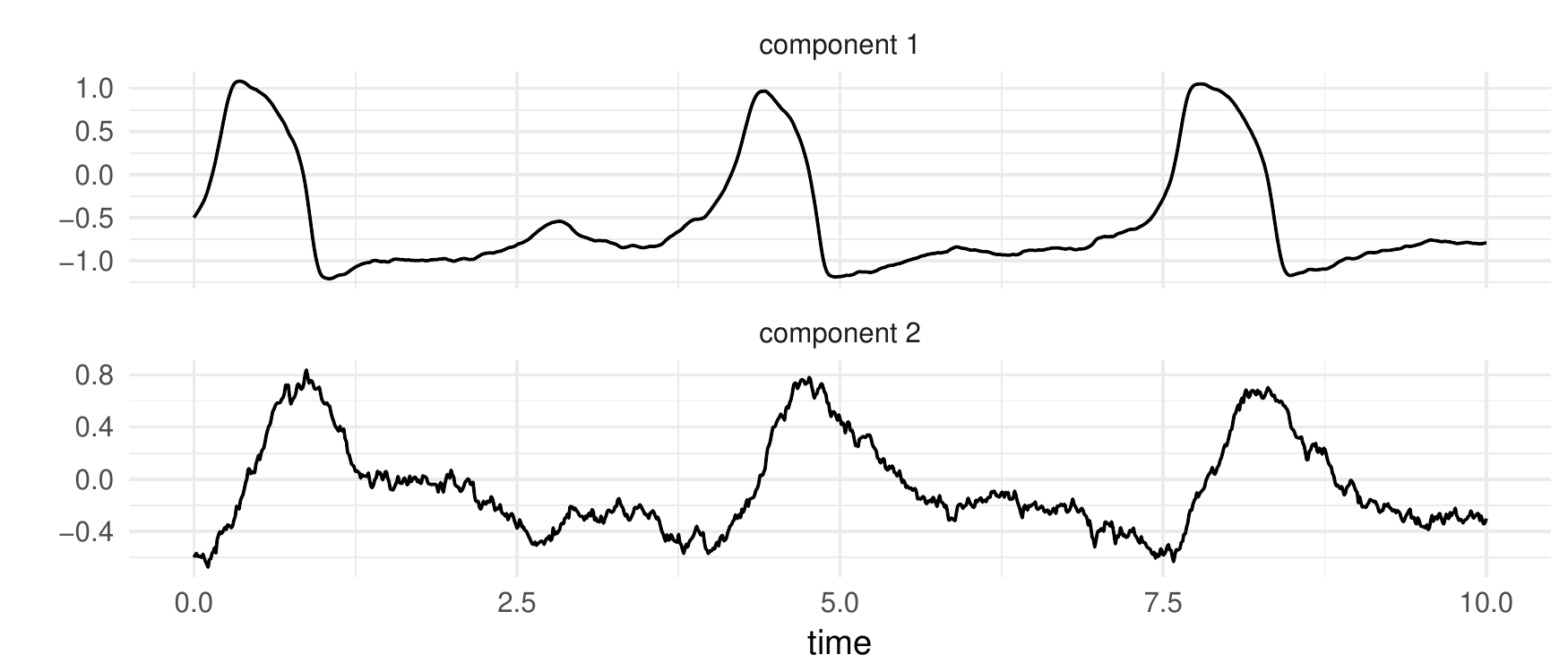}
	\end{center}
	\caption{A realisation of a sample path of the FitzHugh-Nagumo model as specified in Equation \eqref{eq:fitz}, with parameter values as in \eqref{eq:fitz-parvalues}. \label{fig:fh-samplepath}}
\end{figure}

While this example formally does not fall into our setup, the conditions of Assumption \ref{ass:sigma_const} strongly suggest that  the component of the drift with smooth path, i.e.\ the first component of $b$, certainly needs to match at the observed endpoint. 
We construct guided proposals by linearising the drift term $X_{t,1}$ at the observed endpoint $v$. Hence,  using that  $-x^3\approx 2 a^3 -3 a^2 x$ for $x$ near $a$, we take
\[ \tilde{B}(t)=  \Bm 1/\eps-3 v^2/\eps & -1/\eps \\ \ga & -1 \Em, \quad \tilde\beta(t) = \Bm 2v^3/\eps +s/\eps \\ \beta\Em
\quad \tilde\sigma= \Bm 0 \\ \si\Em. \]

To illustrate the performance of our method, we take a rather challenging, strongly nonlinear problem. We consider bridges over the time-interval $[0,T]$ with $T=2$, starting at $x(0)=\Bm -0.5& -0.6\Em$. In Figure \ref{fig:forwards} we forward simulated $100$ paths, to access the behaviour of the process.
Next, we consider two cases:
\begin{enumerate}
  \item[(a)] Conditioning on the first coordinate at the endpoint of a ``typical'' path; we took $v=-1$.
  \item[(b)] Conditioning on the first coordinate at the endpoint of an ``extreme'' path; we took $v=1.1$.
\end{enumerate}
We ran the sampler for $50.000$ iterations, using $\rho=0$ and $\rho=0.9$ in cases (a) and (b) respectively. The percentage of accepted proposals in the  Metropolis-Hastings step equals $64\%$ and $21\%$ respectively.   In Figures \ref{fig:fh-typical} and \ref{fig:fh-extreme} we plotted every $1000$-th sampled path out of the $50.000$ iterations for the ``typical'' and ``extreme'' cases respectively. 
Figure \ref{fig:fh-typical} immediately demonstrates that  for a typical path, guided proposals very closely resemble true bridges (using Figure \ref{fig:forwards} as comparison). 
To assess whether in the ``extreme'' case the sampled bridges resemble true bridges, we also forward simulated the process, only keeping those paths for which $|L x_T-v| <0.01$. The resulting paths are shown in Figure \ref{fig:bridges_extreme_blindforward} and resemble those in Figure  \ref{fig:fh-extreme} quite well. 

This example is extremely challenging in the sense that we take a rather long time horizon ($T=2$), the noise-level on the second coordinate is small and the drift of the diffusion is highly nonlinear. As a result,  the true distribution of bridges is multimodal. Even in much simpler settings, sampling form a multimodal distribution using MCMC constitutes a difficult problem. Here, the multimodality  is recovered remarkably well by our method as can be seen from Figure \ref{fig:fh-extreme}. 

\begin{figure}
	\begin{center}
	\includegraphics[scale=0.7]{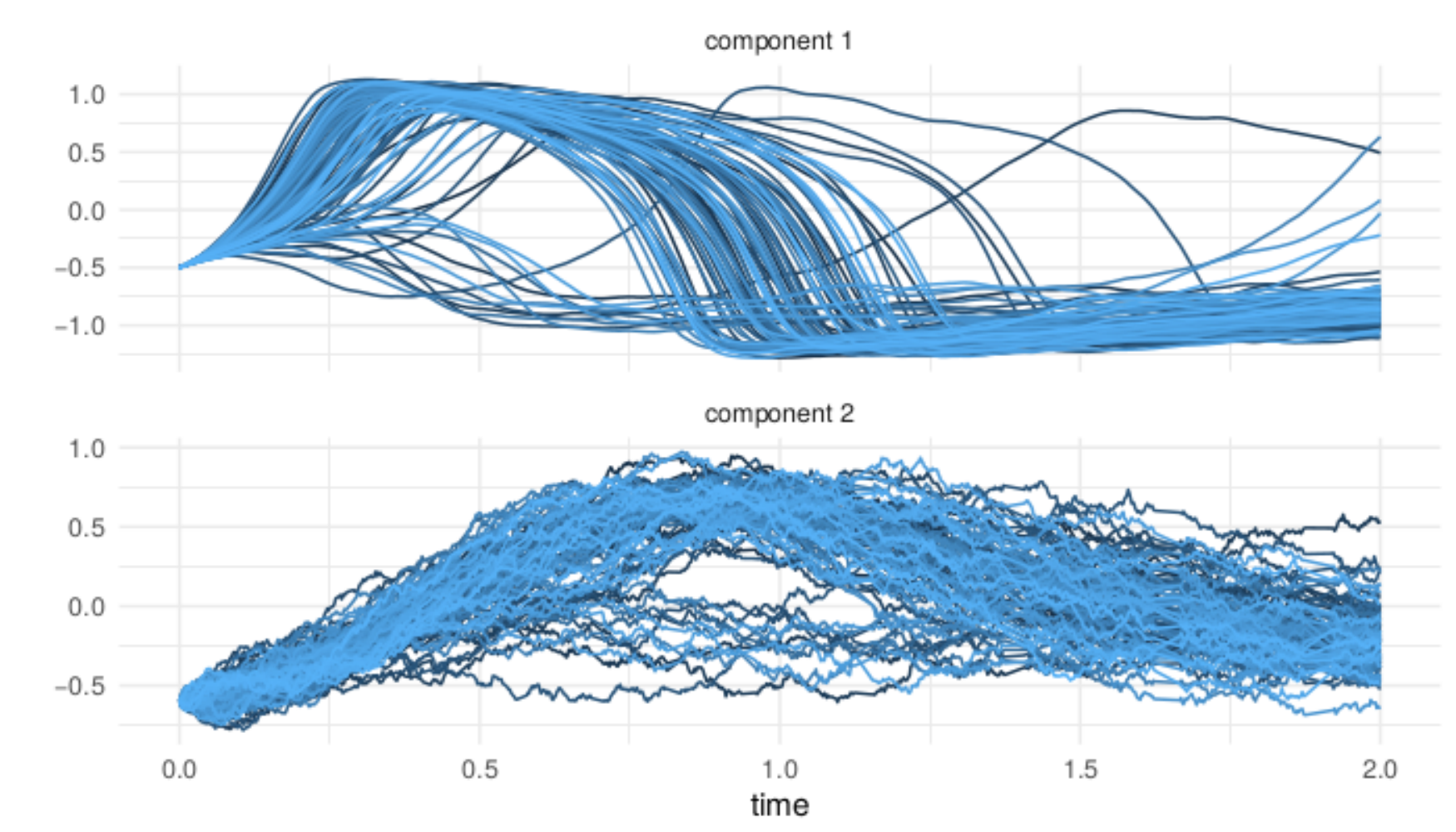}
	\end{center}
		\caption{Realisations of $100$ forward sampled paths for the FitzHugh-Nagumo model as specified in Equation \eqref{eq:fitz}, with parameter values as in \eqref{eq:fitz-parvalues}.  \label{fig:forwards}}

\end{figure}

\begin{figure}
	\begin{center}
		\includegraphics[scale=0.7]{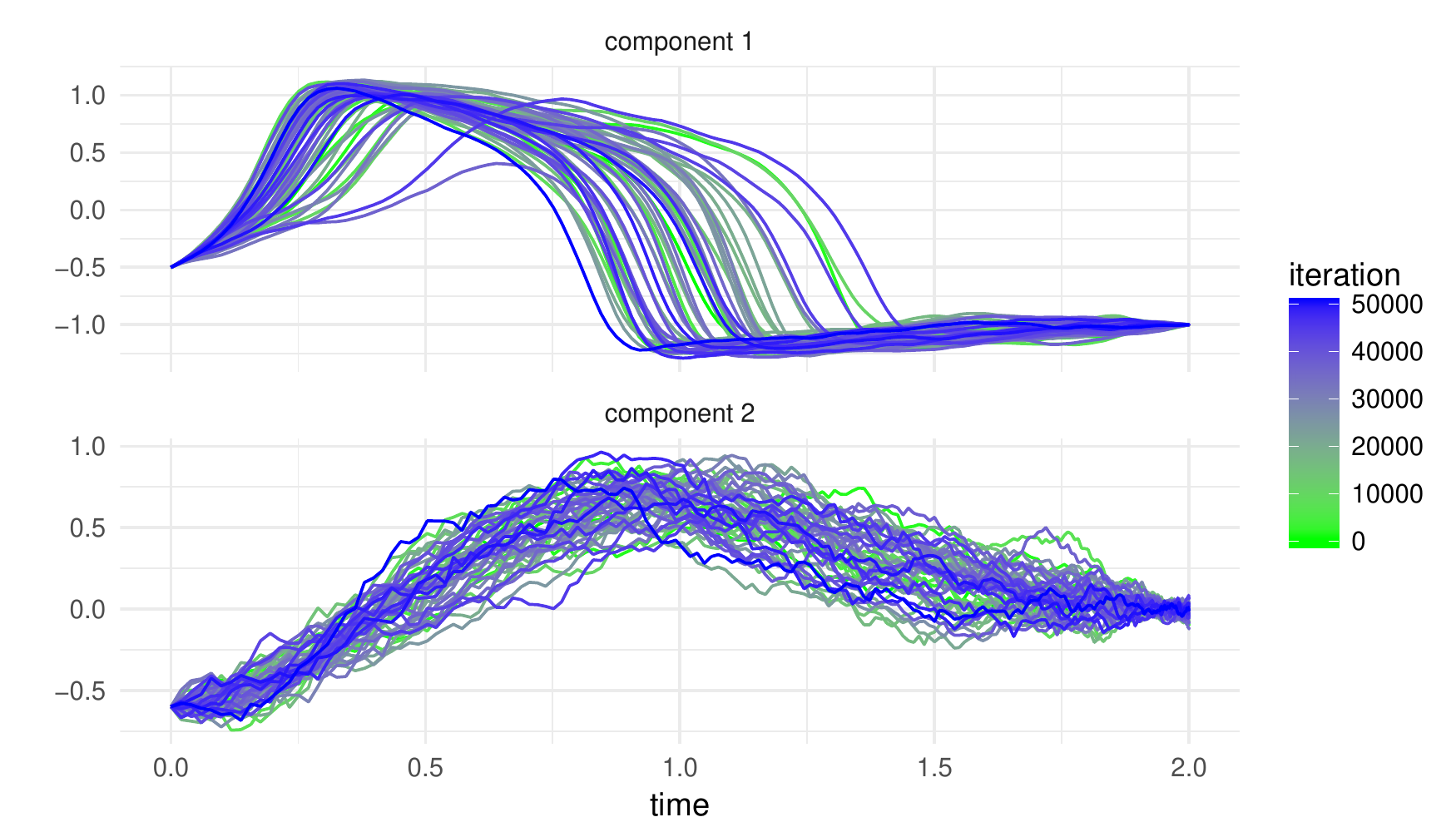}
	\end{center}
	\caption{Sampled guided diffusion bridges when conditioning on $v=-1$ (typical case).  \label{fig:fh-typical}}
\end{figure}

\begin{figure}
	\begin{center}
		\includegraphics[scale=0.7]{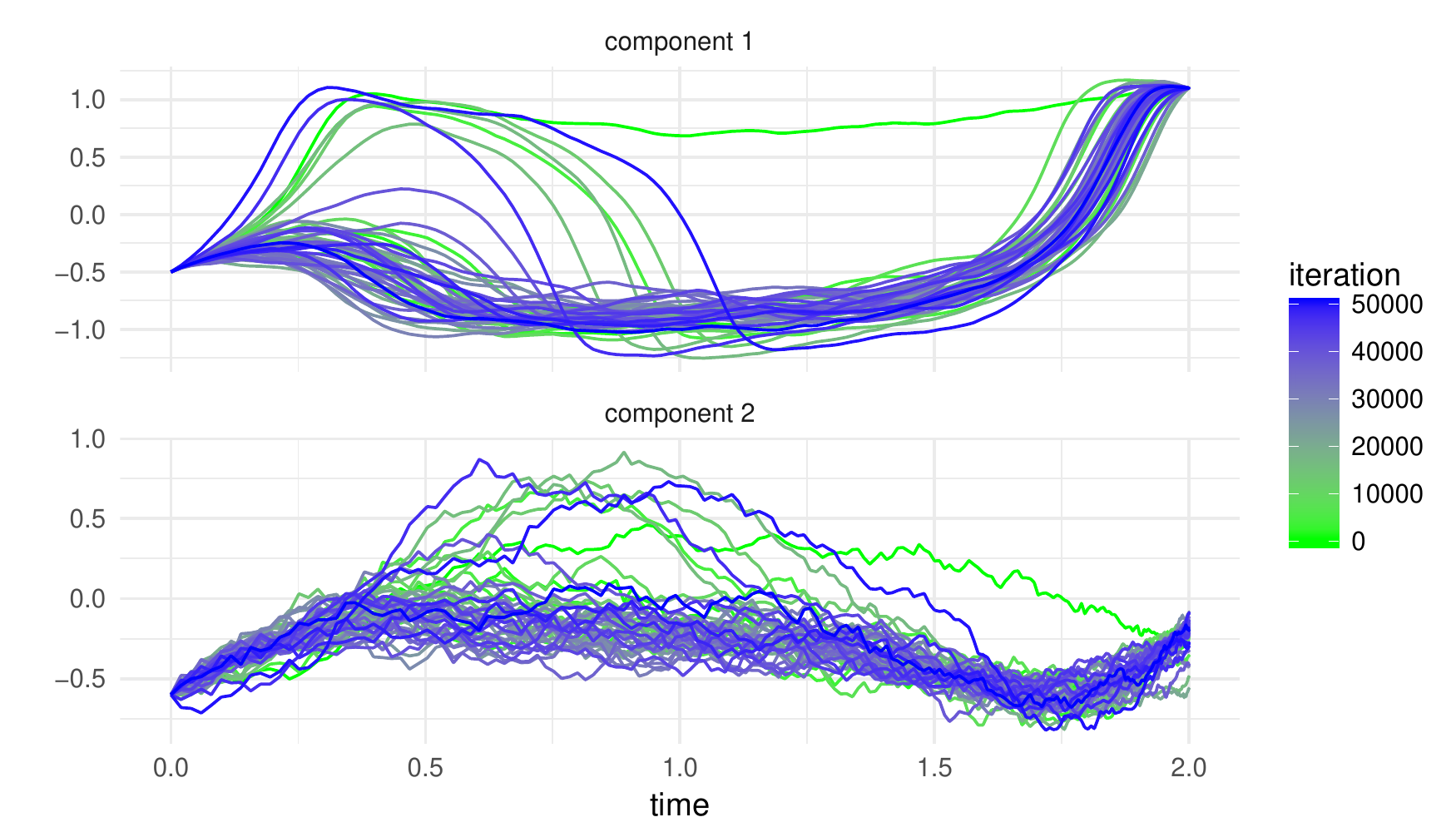}		
	\end{center}
	\caption{Sampled guided diffusion bridges when conditioning on $v=1.1$ (extreme case). The ``outlying'' green curve corresponds to the initialisation of the algorithm. 
	\label{fig:fh-extreme}}
\end{figure}

\begin{figure}
	\begin{center}
	\includegraphics[scale=0.7]{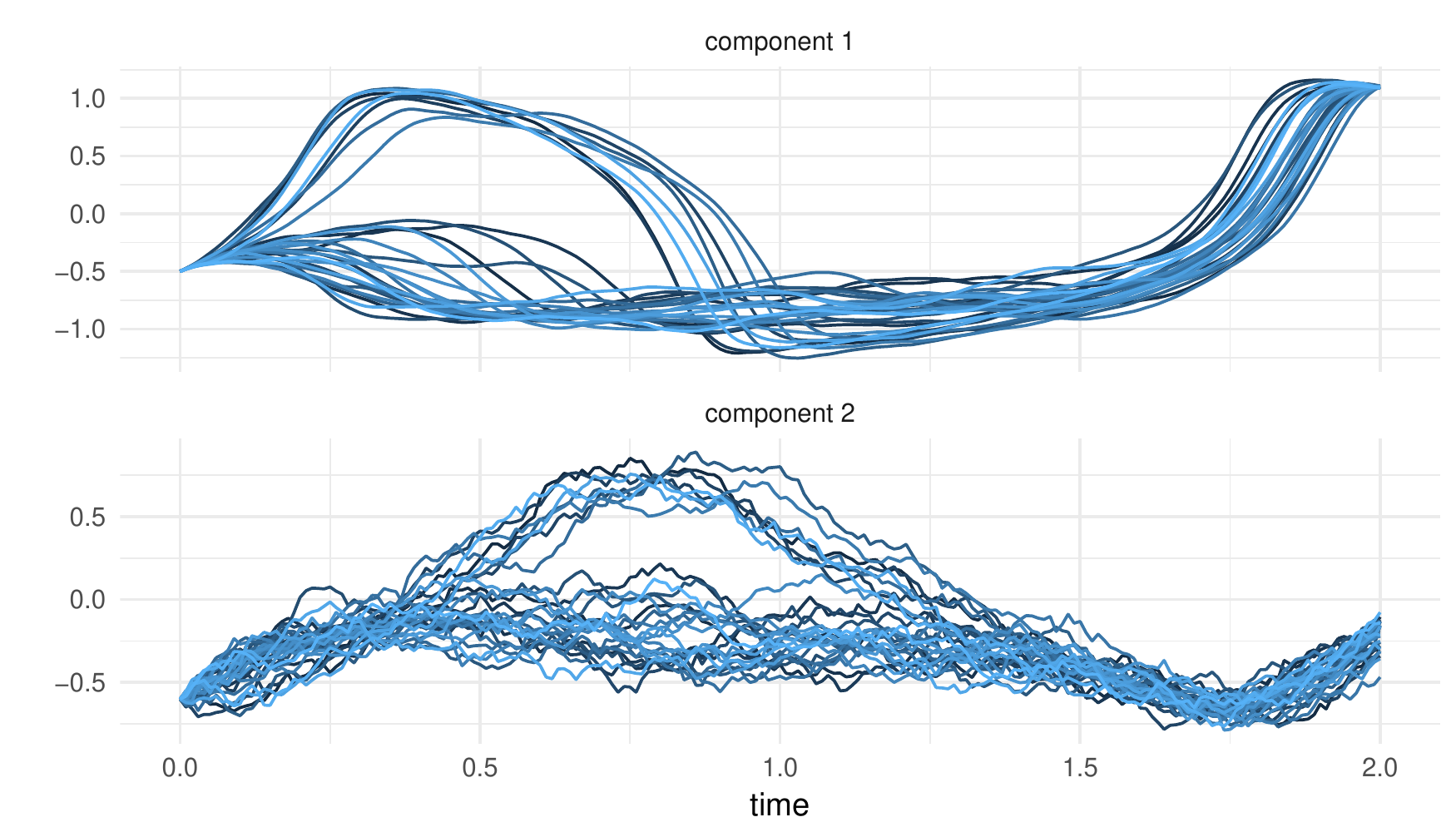}
	\end{center}
		\caption{Realisations of $30$ forward sampled paths for the FitzHugh-Nagumo model as specified in Equation \eqref{eq:fitz}, with parameter values as in \eqref{eq:fitz-parvalues}. Only those paths are kept for which  $|L x_T-v| <0.01$, where $v=1.1$ (the conditioning for the ``extreme'' case).   \label{fig:bridges_extreme_blindforward}}
\end{figure}
	
\end{ex}

\begin{rem}
We have chosen for $50.000$ iterations in the chosen examples. However, qualitatively the same figures of simulating bridges can be obtained by reducing the  number of iterations to approximately  $10.000$.
\end{rem}

\subsection{Numerical checks on the validity of guided proposals}
In this section we first investigate the quality of guided proposals over long time spans. Next, we empirically demonstrate that the conditions of our main theorem, especially Assumption \ref{ass:sigma_const}, is stronger than actually needed. In each numerical experiment we compare two histogram estimators for $v \mapsto \rho(0,x_0; T,v)$. The first estimator is obtained by making a histogram of a large number of forward simulations of the unconditioned diffusion process. Denote by $\{A_k\}$ the bins of this histogram.   A second estimator is obtained by using the equality
\[ \rho(0,x_0; T, v) =   \tilde \rho(0, x_0; T, v) \expec{\Psi_T(X^{\circ, T, v}) }
\]
which is a direct consequence of Theorem \ref{thm:main}. 
Note that we  extended the notation to highlight that  $\rho$, $\tilde \rho$ and $X^\circ$ depend on $T$ and $v$.
We use the relation in the previous display as follows: for each bin $A_k$
\[ 
\int_{A_k} \rho(0,x_0; T, v) \dd v =\expec{\ind_{A_k}(\tilde V) \frac{\tilde \rho(0, x_0; T, \tilde V)}{q( \tilde V)}  \Psi_T(X^{T, \tilde V,\circ}) }
\]
where $\tilde V$ is sampled from the density $q$. 
Hence, $\int_{A_k} \rho(0,x_0; T, v) \dd v$ can be approximated using importance sampling where repeatedly first the endpoint $v$ is sampled from $q$ and subsequently  a guided proposal is simulated that is conditioned to hit $v$ at time $T$. In our experiments we took the importance sampling density $q$ to be the  Gaussian density with mean and covariance obtained from the unconditioned forward simulated endpoint values. 

Note that the setup is such that this is feasible, at least when estimating the entire histogram, but of course it would be prohibitively expensive to use forward sampling to compute the density in a single small bin or at a single point.

\begin{ex}
Consider the non-linear hypo-elliptic 2d system
determined by drift $b(t, x) = B x \,+\, \beta(t) \,+\, \Bm 0; \, \tfrac12 \sin(x_2) \Em$
with $B =\tfrac{1}{10} \Bm -1 & 1; 0 & {-1}\Em $, $\beta(t) = \Bm 0 & \tfrac12\sin(t/4) \Em$, and dispersion $\sigma \equiv \Bm 0;\; 2\Em$ (with a semicolon separating matrix rows).
Starting at $X_0 = \Bm 0 ;\,  {-\pi/2}\Em$,  we assume to observe $V = L X_T + Z$ with  $L = \Bm 1 & 1\Em$, 
 $Z \sim N(0, 10^{-6})$. We consider both $T=4\pi$ and $T=40\pi$, the latter to check how guided proposals perform over a very long time span.
We take guided proposals derived from  $\tilde b(t, x) = B x + \beta(t)$ and $\tilde\sigma=\sigma$. 

In  Figure \ref{fig:dens-est}  the two histograms are contrasted.  Interestingly, the results show no degradation in performance when increasing $T$ by an order.
For the simulations we took $K = 70$ bins and $100\,000$ samples of $V$ respective draws from $\tilde V$ (thus on average approximately $1\,500$ draws per bin) and time grid $t_i =  s_i (2 -  s_i/T)$ with $s_i = h \, i$, $h= 0.01$, therefore decreasing step-size towards $T$ while keeping the number of grid points equal to $T/h$, as suggested in \cite{vdm-s-estpaper}. The implementation is based on our Julia package \cite{BridgeSDEInference} with package co-author Marcin Mider. The figures also serve to verify the correctness of the implementation.
\end{ex}

\begin{ex}\label{example:sigma(x)}
It is interesting to ask if -- numerically speaking -- the change of measure is successful in cases where  $\sigma$ depends on $x$ and the fourth inequality of Assumption \ref{ass:controll} cannot be verified. For that purpose, we slightly adjust the setting of the previous example by now taking $L = \Bm 1 & 0\Em$ and 
$\sigma(t, x) = \Bm 0; 2 + \tfrac12\cos(x_2)  \Em$ and  repeating the experiment.
In this case we chose $\tilde \sigma = \sigma(0,  \Bm 0 ; 0\Em)$. 
As the problem is more difficult, we took less bins ($K = 50$)  and set $h = 0.005$ (otherwise keeping our previous choices.)
The resulting Figure \ref{fig:dens-est2} shows  no indication of lack of absolute continuity or loss of probability mass. This strongly indicates that guided proposals can perform perfectly fine for the present complex setting that includes state-dependent diffusion coefficient and hypo-ellipticity. 

However, care is needed, in Figure \ref{fig:dens-est3} we show the result for the same experiment, but with $L$ changed to  $L = \Bm 1 & 1\Em$. Here,  the loss of probability mass indicates violation of absolute continuity. We conjecture that $L a(T,v) L^\T = L \tilde{a}(T) L^\T$ may be the ``right'' restriction on choosing $\tilde{a}(T)$. To obtain empirical evidence, we redid the experiment with $L = \Bm 1 & 1\Em$ but now $\sigma(t, x) = \Bm 0; 2 + \tfrac12\cos(Lx)  \Em$. In this case one can match the diffusivity at time $T$ by taking $\tilde\sigma = \Bm 0; 2 + \tfrac12\cos(v)  \Em$. The resulting figure (Figure \ref{fig:dens-ll1match}) indicates no loss of absolute continuity, supporting the conjecture.

\end{ex}

\begin{figure}
\begin{center}
\includegraphics[width=0.7\linewidth]{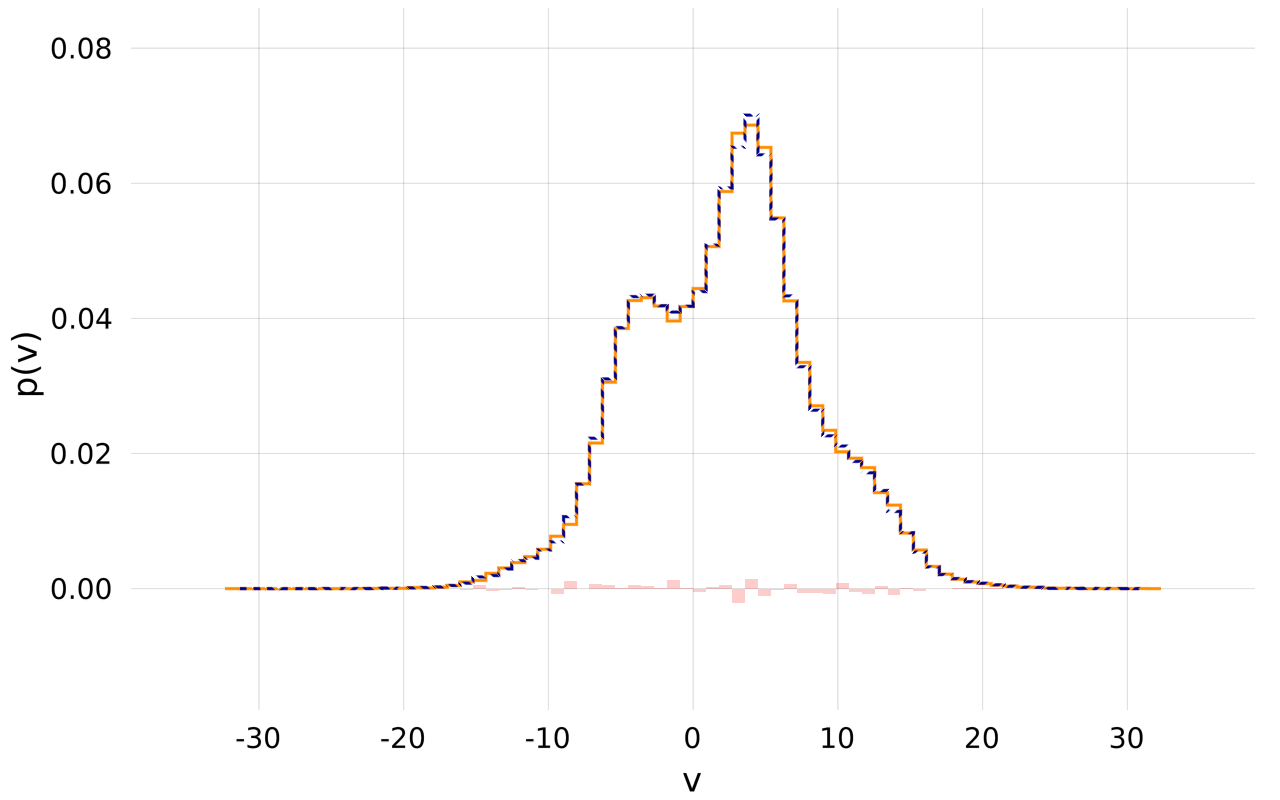}

\includegraphics[width=0.7\linewidth]{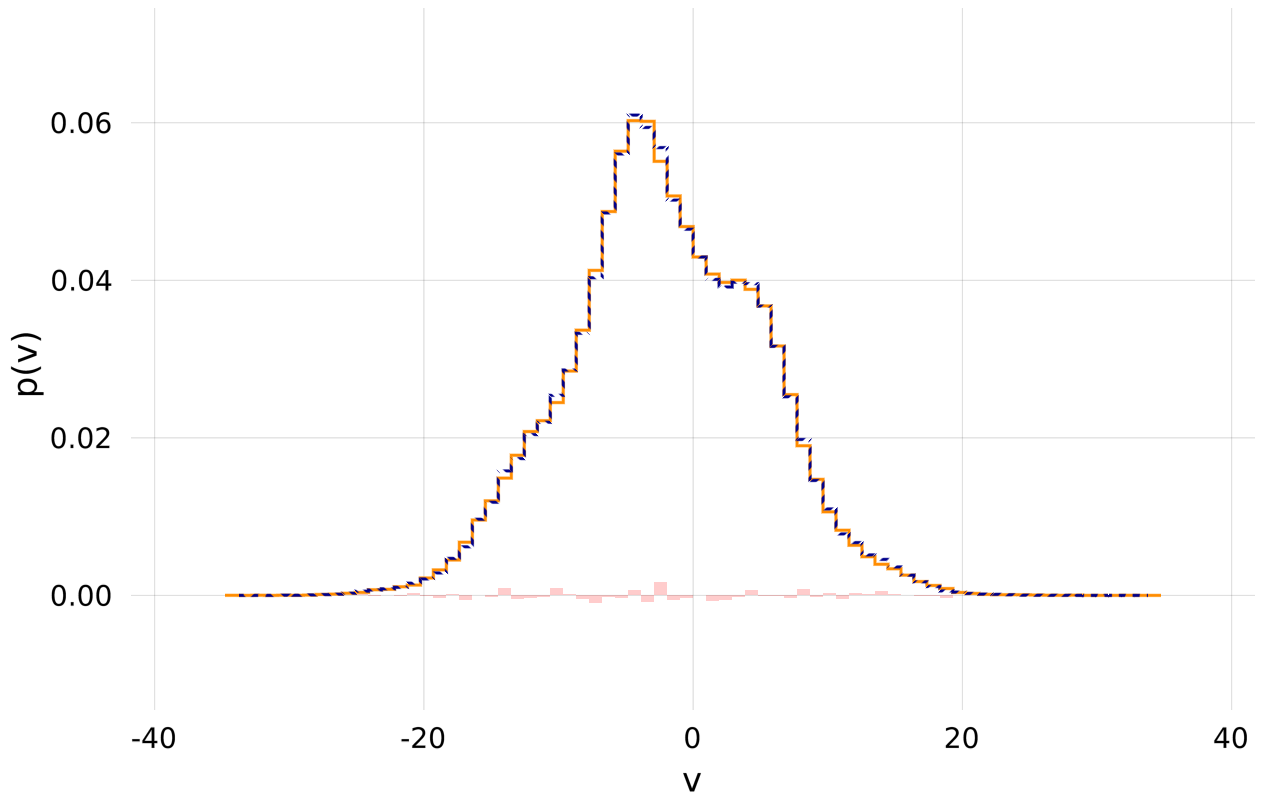}
\caption{Dark orange: Histogram baseline estimate of the density of observation  $V = L X_T + Z$, $Z \sim N(0, 10^{-6})$ from forward simulation. Dashed blue: observation density estimate using weighted histogram 
of points $\tilde V$ sampled from Gaussian distribution weighted with
importance weights from guided proposals steered towards those points. Top: $T = 4\pi$. Bottom: $T = 40\pi$.  Pink: difference between histograms.}
\label{fig:dens-est}
\end{center}

\end{figure}

\begin{figure}
\begin{center}
\includegraphics[width=0.7\linewidth]{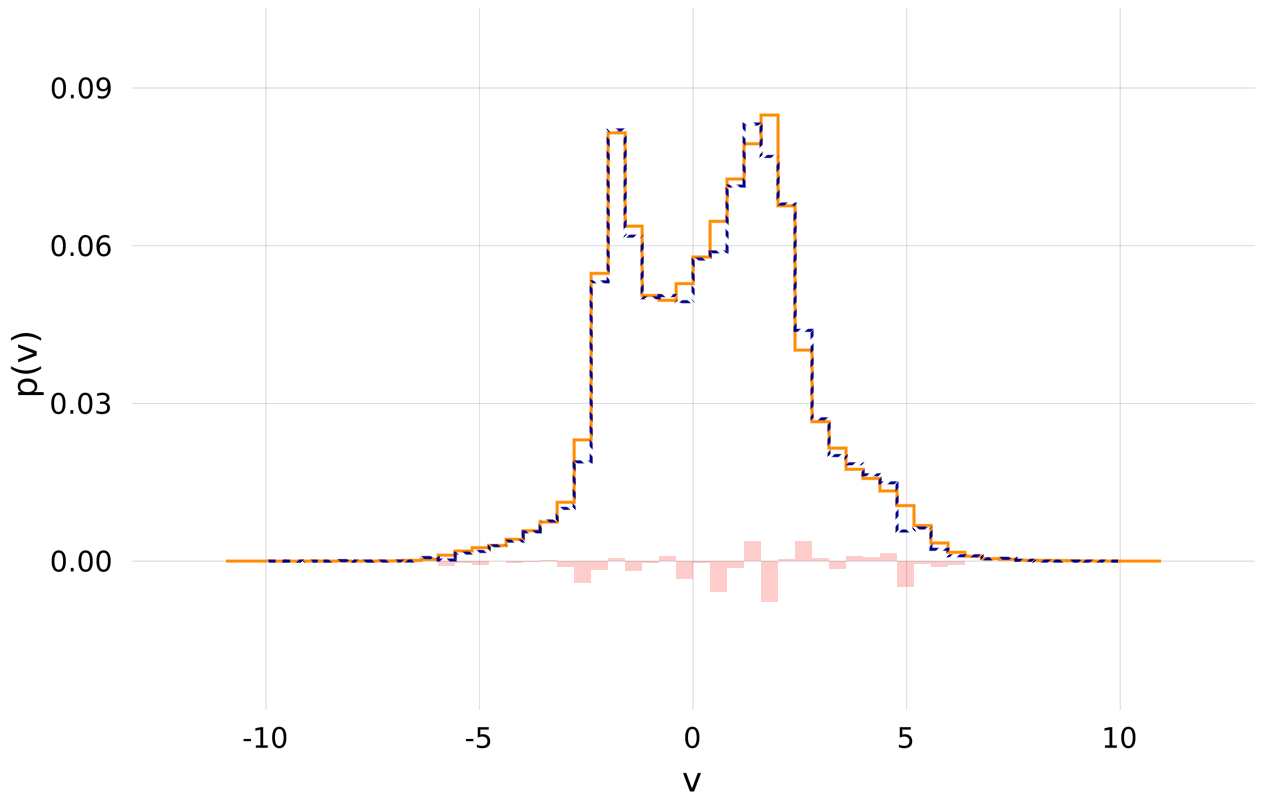}
\caption{As Figure \ref{fig:dens-est}, but estimates for the model with observation operator $L =\left[1 \; 0\right]$ and $\sigma(t, x) = \left[ 0; 2 + \tfrac12\cos(x_2) \right]$, at $T = 4\pi$.}
\label{fig:dens-est2}
\end{center}

\end{figure}

\begin{figure}
\begin{center}
\includegraphics[width=0.7\linewidth]{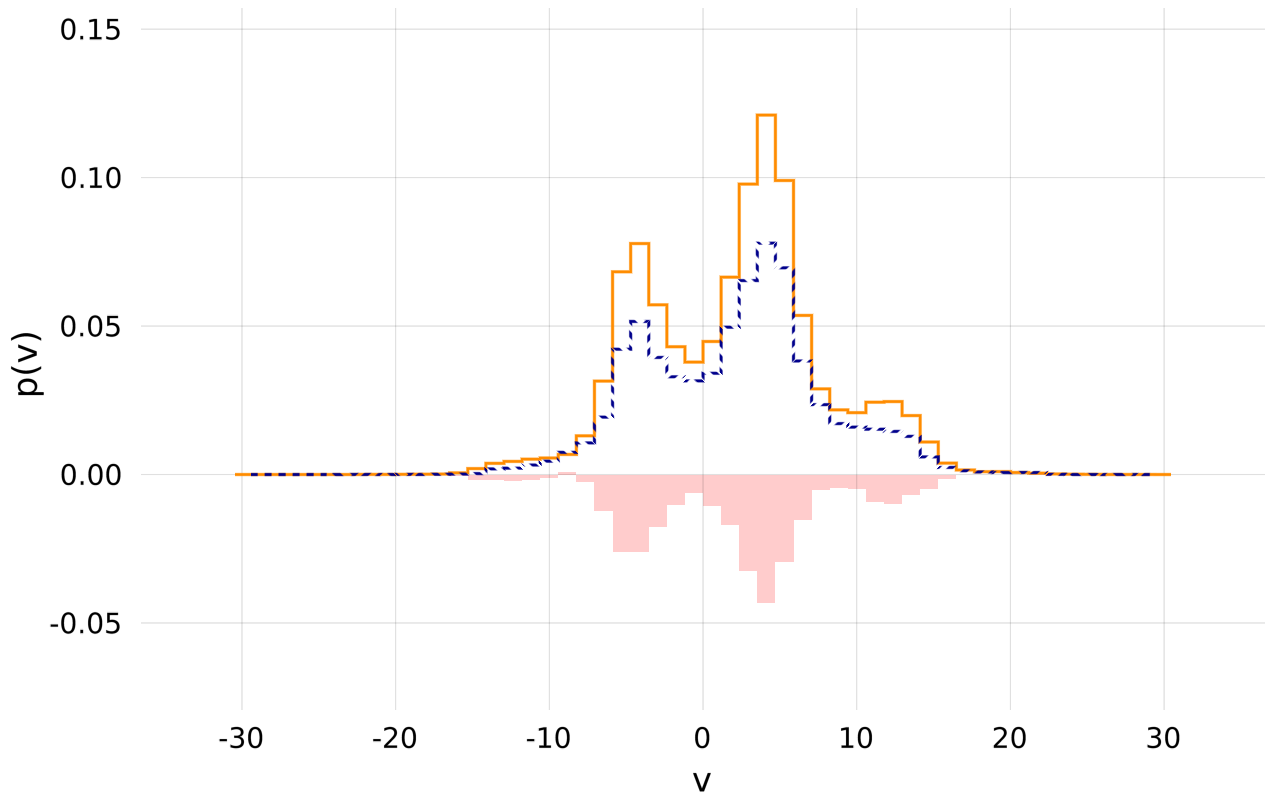}
\caption{As Figure \ref{fig:dens-est}, but estimates for the model with observation operator $L =\left[1 \; 1\right]$ and $\sigma(t, x) = \left[ 0; 2 + \tfrac12\cos(x_2) \right]$, at $T = 4\pi$.
Note the loss of probability mass indicating lack of absolute continuity.
}
\label{fig:dens-est3}
\end{center}

\end{figure}

\begin{figure}
\begin{center}
\includegraphics[width=0.7\linewidth]{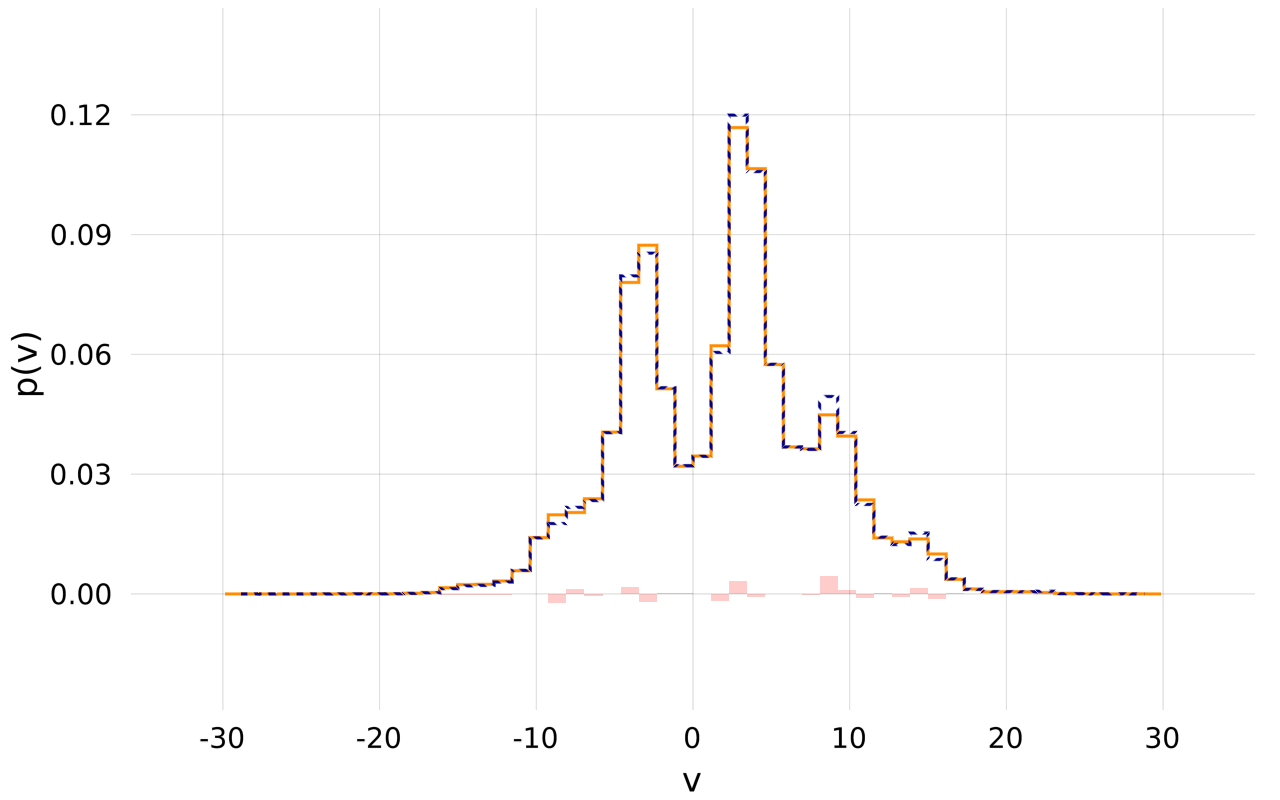}
\caption{As Figure \ref{fig:dens-est}, but estimates for the model with observation operator $L =\left[1 \; 1\right]$ and $\sigma(t, x) = \left[ 0; 2 + \tfrac12\cos(x_1 + x_2) \right]$, at $T = 4\pi$.
}
\label{fig:dens-ll1match}
\end{center}

\end{figure}

\section{Proofs of Proposition \ref{thm:limitU} and Corollary \ref{cor:limit-behav-ue}}
 \label{sec:behaviour-near-endpoint}
 
 In this section we give proofs of the results from Section \ref{sec:behav} on the behaviour of guided proposals near the conditioning point. 
For clarity, the proof   of Proposition \ref{thm:limitU} is split up over subsections \ref{subsec:centringscaling}, \ref{subsec:recap} and \ref{subsec:prooflimiting}. The proof of  Corollary \ref{cor:limit-behav-ue}  is in section \ref{subsec:proofcor}. 
\subsection{Centring and scaling of the guided proposal} \label{subsec:centringscaling}
To reduce notational overhead, we write  $a_t\equiv a(t,X^\circ_t)$. Then $\tilde{b}_t$, $b_t$ and $\si_t$ are defined similarly. Our starting point is the expression for $\tilde{r}$ in \eqref{eq:tilder}. 

\begin{lem}\label{lem:sde-Z}
If we define
\[	Z_t = v- \mu(t)-L(t) X^\circ_t ,\] then
\[ \dd Z_t = L(t) \left(\tilde{b}_t-b_t\right)\dd t + L(t)\si_t \dd W_t  - L(t)a_t L(t)^\T M(t) Z_t \dd t.
\]
\end{lem}
\begin{proof}
We have
\[ \dd Z_t = -\left(\frac{\dd}{\dd t} L(t)\right) X^\circ_t \dd t -\frac{\dd}{\dd t} \mu(t)-L(t) \dd X^\circ_t. \]
The results now follows because the first two terms on the right-hand-side together equal $L(t) \tilde{b}(t,X^\circ_t)$.
\end{proof}

\begin{lem}\label{lem:ZMZ}
We have
\begin{align*}
	\frac12 \dd \left(Z_t^\T M(t) Z_t\right)  = &\frac12 Z_t^\T M(t) L(t) \left(\tilde{a}(t)-a_t\right) L(t)^\T M(t) Z_t \dd t \\ & +
	Z_t^\T M(t) L(t) \left(\tilde{b}_t-b_t\right) \dd t  \\& + Z_t^\T M(t) L(t) \si_t \dd W_t -\frac12 Z_t^\T M(t) L(t) a_t L(t)^\T M(t) Z_t \dd t \\ & +
	\tr\left(L(t) a_t L(t)^\T M(t)\right)\dd t.  
\end{align*}
\end{lem}
\begin{proof}
By Ito's lemma
\[
	\frac12 \dd \left(Z_t^\T M(t) Z_t\right)  = \frac12 Z_t^\T \frac{\dd M(t)}{\dd t} Z_t \dd t  +
	Z_t^\T M(t) \dd Z_t   +
	\tr\left(L(t) a_t L(t)^\T M(t)\right)\dd t.  
\]
Next, substitute the SDE for $Z_t$ from lemma
\ref{lem:sde-Z} and use that
\[ \frac{\dd M(t)}{\dd t} = -M(t) \frac{\dd M(t)^{-1}}{\dd t} M(t) =M(t) L(t) \tilde{a}(t) L(t)^\T M(t).\]
The final equality follows from the fact that $\Md(t)=M(t)^{-1}$ satisfies the ordinary differential equation 
$\dd \Md(t)=- L(t) \tilde{a}(t) L(t)^\T\dd t.$
 The result follows upon reorganising terms. 
\end{proof}

Whereas in the uniformly elliptic case all elements of $Z_t$ and $M(t)$ behave in the same way as a function of $T-t$, this is not the case in the hypo-elliptic case. For this reason, we introduce a diagonal scaling matrix $\Delta(t)$. 
\begin{lem}\label{eq:U'PU}
Let $\Delta(t)$ be an invertible $m\times m$ diagonal matrix. If  $\ZD{t}$, $\LD(t)$ and $\MD(t)$ are as defined  in Equations \eqref{eq:defU} and \eqref{eq:def-tildeC-P} then 
\begin{equation}\label{eq;sde-quadratform-U}
\begin{split}
	\frac12 \dd &\left(\ZD{t}^\T \MD(t) \ZD{t}\right)  = \frac12 \ZD{t}^\T \MD(t) \LD(t) \left(\tilde{a}(t)-a_t\right) \LD(t)^\T \MD(t) \ZD{t} \dd t \\ & +
	\ZD{t}^\T \MD(t) \LD(t) \left(\tilde{b}_t-b_t\right) \dd t  \\& + \ZD{t}^\T \MD(t) \LD(t) \si_t \dd W_t -\frac12 \ZD{t}^\T \MD(t) \LD(t) a_t \LD(t)^\T \MD(t) \ZD{t} \dd t \\ & +
	\tr\left(\LD(t) a_t \LD(t)^\T \MD(t)\right)\dd t.  
\end{split}
\end{equation}
Moreover, 
\begin{equation}\label{eq:tilder-PC}
	 \tilde{r}(t,X^\circ_t)=\LD(t)^\T \MD(t)
\ZD{t} \end{equation} 
\end{lem}
\begin{proof}
	This is a straightforward consequence of Lemma \ref{lem:ZMZ}. The expression for $\tilde{r}$ follows from equation \eqref{eq:tilder}.
\end{proof}

\subsection{Recap on notation and results}\label{subsec:recap}
 For clarity we summarise our notation, some of which was already defined in Section \ref{sec:notation}. The auxiliary process is defined by the SDE $\dd \tilde{X}_t = (\tilde{B}(t) \tilde{X}_t + \tilde\beta(t)) \dd t + \tilde\si(t) \dd W_t$. The matrix $\Phi(t)$ satisfies the ODE $\dd \Phi(t) =\tilde{B}(t) \Phi(t) \dd t$ and we set $\Phi(T,t)=\Phi(T) \Phi(t)^{-1}$. A realisation $v$ of $V=LX_T$ is observed. 
The scaled process is defined by 
\[\ZD{t}=\Delta(t) \left( v- \int_t^T L(\tau) \tilde\beta(\tau) \dd \tau-L(t) X^\circ_t  \right),\]
where
$ L(t)= L \Phi(T,t) \qquad \text{and} \qquad 
 \LD(t)=\Delta(t)L(t). $
 \sloppy Furthermore, we defined 
$  M(t)= \left(\int_t^T L(\tau) \tilde{a}(\tau) L(\tau)^\T \dd \tau\right)^{-1} \qquad \text{and}\qquad  \MD(t) = \Delta(t)^{-1} M(t) \Delta(t)^{-1},$
where $\tilde{a}(t)=\tilde\si(t) \tilde\si(t)^\T$. 
Finally, the guiding term in the SDE for the guided proposal $X^\circ_t$ is given by 
$a(t,X^\circ_t) \LD(t)^\T \MD(t) \ZD{t}$. 
The process $t\mapsto \ZD{t}$ is the key object to be studied in this section.

\subsection{Proof of Proposition \ref{thm:limitU}}\label{subsec:prooflimiting}

The line of proof is exactly as suggested in \cite{mao1992} (page 341):
\begin{enumerate}
  \item Start with the Lyapunov function $V(t,\ZD{t})=\tfrac12 \ZD{t}\T \MD(t) \ZD{t}$.
  \item Apply Ito's lemma to $V(t,\ZD{t})$.
  \item Use martingale inequalities to bound the stochastic integral.
  \item Apply a Gronwall type inequality.
\end{enumerate}

We bound all terms appearing in equation \eqref{eq;sde-quadratform-U}. Note that the first term on the right-hand-side vanishes. We start with the Wiener integral term. To this end, fix $t_0 \in [0,T)$ and 
 let 
\[ N_t= \int_{t_0}^t \ZD{s}^\T \MD(s) \LD(s) \si \dd W_s.\]
Then 
\[ \int_{t_0}^t \ZD{s}^\T \MD(s) \LD(s) \si_s \dd W_s -\frac12 \int_{t_0}^t \ZD{s}^\T \MD(s) \LD(s)\, a_s\, \LD(s)^\T \MD(s) \ZD{s} \dd s =N_t -\frac12[N]_t.\]
Now $N_t$ can be bounded using an exponential martingale inequality. 
Let $\{\ga_n\}$ be a sequence of positive numbers. Define for $n\in \NN$, $t_n=T-1/n$ and 
\[ E_n=\left\{ \sup_{0\le t \le t_{n+1}} \left( N_t- \frac{1}{2} [N]_t \right) > \ga_n\right\}. \]
By the exponential martingale inequality of Theorem 1.7.4 in \cite{Mao}, we obtain that $P(E_n) \le \e^{-\ga_n}$. 
If we assume $\sum_{n=1}^\infty \e^{-\ga_n}<\infty$, then by the Borel-Cantelli lemma 
$ P\left(\limsup_{n\to \infty} E_n\right)=0. $
Hence, for almost all $\omega$, $\exists\, n_0(\omega)$ such that for all $n\ge n_0(\omega)$
\begin{equation}\label{eq:martingale-partt} \sup_{t_0\le t \le t_{n+1}} \left( N_t -\frac{1}{2}\int_0^t [N]_t  \right)\le\ga_n. \end{equation}
Let $\eps>0$. Upon taking $\ga_n=(1+2\eps) \log n$ we get $\sum_{n=1}^\infty \e^{-\ga_n}=\sum_{n=1}^\infty n^{-1-2\eps} <\infty$.
Since $\MD(t)$ is strictly positive de\-fi\-nite
\[ \lmin(\MD(t)) \|\ZD{t}\|^2 \le  \ZD{t}^\T \MD(t) \ZD{t}. \]
Assume $t_0 < t< t_{n+1}$. Combining the inequality of the preceding display  with Lemma \ref{eq:U'PU} and substituting the bound in \eqref{eq:martingale-partt}, we obtain that for any $\eps>0$
\begin{align*} \frac12\lmin(\MD(t)) \|\ZD{t}\|^2 &\le \frac12 \ZD{t_0} \MD(t_0) \ZD{t_0}\\&+\int_{t_0}^t \|\ZD{s}\| \left\|  \MD(s)\right\| \left\| \LD(s) \left(\tilde{b}_s-b_s\right)\right\| \dd s \\ & \qquad +\ga_n +
\int_{t_0}^t \tr\left(\LD(s)\, a_s\, \LD(s)^\T  \MD(s)\right) \dd s \\ & \qquad +  \frac12 \int_{t_0}^t\ZD{s}^\T \MD(s) \LD(s) \left(\tilde{a}(s)-a_s\right) \LD(s)^\T \MD(s) \ZD{s} \dd s 
\end{align*}	
Recall that  for po\-sitive semi\-de\-finite matrices $A$ and $C$ we have $|\tr(AC)| \le \tr(A) \tr(C)\le \tr(A) p \la_{\rm max}(C)$, if $C \in \RR^{p\times p}$.
Hence, 
\begin{equation}\label{eq:boundtraceterm} \tr\left(\LD(s)\, a_s\, \LD(s)^\T  \MD(s)\right) \le 
\tr\left(\LD(s)\, a_s\, \LD(s)^\T\right)  m  \la_{\rm max}(\MD(s))
\end{equation}
Furthermore, as \begin{equation}\label{eq:boundnorm-by-lmax}
\|\MD(s)\|=\sqrt{ \la_{\rm max}(\MD(s)^2)} =  \la_{\rm max}(\MD(s))\end{equation} we can combine the preceding three inequalities to obtain
\begin{equation}\label{eq:start-gronwall-argument}
\begin{split}
 \frac12\lmin(\MD(t)) \|\ZD{t}\|^2 &\le \frac12 \ZD{t_0} \MD(t_0) \ZD{t_0}\\&+\int_{t_0}^t \|\ZD{s}\|  \la_{\rm max}(\MD(s)) \left\| \LD(s) \left(\tilde{b}_s-b_s\right)\right\| \dd s \\ & \qquad +\ga_n + m
\int_{t_0}^t \tr\left(\LD(s)\, a_s\, \LD(s)^\T\right)\la_{\rm max}(\MD(s))  \dd s\\ & \qquad +  \frac12 \int_{t_0}^t\ZD{s}^\T \MD(s) \LD(s) \left(\tilde{a}(s)-a_s\right) \LD(s)^\T \MD(s) \ZD{s} \dd s .
	\end{split}	
\end{equation}

Upon substituting the bounds in \eqref{eq:bounds-thm}
 we get, for certain positive constants $C_0$,  $C_1$, $C_2$, $C_3$ and $C_4$ that 
\begin{equation}\label{eq:bound-ZD}
\begin{split}	
 (T-t)^{-1} \|\ZD{t}\|^2 &\le C_0+ C_1 \int_{t_0}^t \|\ZD{s}\|  (T-s)^{-1} \dd s \\ \qquad &+C_2 \ga_n +
C_3 \int_{t_0}^t (T-s)^{-1} \dd  + C_4 \int_{t_0}^t \|\ZD{s}\|^2 (T-s)^{\alpha-2}.
\end{split}
\end{equation}
 If we define $\xi_t=(T-t)^{-1} \|\ZD{t}\|^2$, then this inequality can be rewritten as
\begin{equation}\label{eq:ineq-xi}
\begin{split} \xi_t &\le C_0 +C_2 \ga_n  + C_3 \log\left(\frac{T-t_0}{T-t}\right)  \\& \qquad +  C_1 \int_{t_0}^t   (T-s)^{-1/2} \sqrt{\xi_s}\dd s + C_4 \int_{t_0}^t (T-s)^{\alpha-1} \xi_s \dd s .
\end{split}
\end{equation}
By Lemma \ref{lem:gronwall2} in the appendix this implies
\begin{align*}
 \xi_t & \leq  \left( \sqrt{C_0 +C_2 \ga_n  + C_3 \log\left(\frac{T-t_0}{T-t}\right)} + \frac12 C_1 \left( \sqrt{T- t_0} - \sqrt{T-t} \right)\right)^2 \\ & \qquad \qquad \qquad  \times  \exp\left(C_4 \int_{t_0}^t (T-s)^{\alpha-1}\right).
 \end{align*}
Now divide both sides of this inequality by $\log(1/(T-t))$ and consider $t_n < t < t_{n+1}$. Then $\log n \le \log(1/(T-t))$.  It then follows that 
\[ \limsup_{t\uparrow T} \frac{\|\ZD{t}\|^2}{(T-t)\log(1/(T-t))} \le C_2(1+2\eps) +C_3.\]
Now let $\eps\downarrow 0$.

\subsection{Proof of Corollary \ref{cor:limit-behav-ue}}\label{subsec:proofcor}
 As $\Delta(t)=I_m$ it is easy to see that  $M(t)=\scr{O}(1/(T-t)$ and $\LD(t)=\scr{O}(1)$. This behaviour of $M(t)$ is also contained in the first inequality of Lemma 8 in \cite{Schauer} (note that in that paper, $\tilde{H}$ corresponds to $M$ as defined in this paper). Now it is easy to see that the conditions of theorem \ref{thm:limitU} are satisfied.

\section{Absolute continuity with respect to the guided proposal distribution}\label{sec:absolute-continuity}

\subsection{Proof of Theorem \ref{thm:main}}

We start with a result that gives the Radon-Nikodym derivative of $\PP^\star_t$ relative to $\PP^\circ_t$ for $t<T$. 
\begin{prop}\label{prop-abscont-beforeT}
For $t<T$ we have
\begin{equation}\label{eq:LR-before-limit} \frac{\dd \PP_t^\star}{\dd \PP_t^\circ}(X^\circ) = \frac{\tilde \rho(0,x_0)}{ \rho(0,x_0)}
\frac{ \rho(t,X^\circ_t)}{\tilde \rho(t,X^\circ_t)} \Psi_t(X^\circ), \end{equation}
where $\Psi_t$ is defined in \eqref{eq:Psi}. 
\end{prop}
\begin{proof}
 Although this result is not a special case of proposition 1 in \cite{Schauer} (where it is assumed that $L=I$ and that the diffusion is uniformly elliptic), the arguments for deriving the likelihood ratio of $\PP^\star_t$ with respect to $\PP^\circ_t$ are the same and therefore omitted. The only thing that needs to be checked is that  $\tilde\rho(t,x)$ satisfies the Kolmogorov backward equation  associated to $\tilde X$.
 This can be proved along the lines of Lemma 3.4 and Corollary 3.5 of \cite{vdm-schauer}. Let $\tilde{\scr{F}}_t =\si\left(\tilde{X}_s,\, 0\le s \le t\right)$ and set $\tilde{Y}_t =\tilde\rho(t,\tilde{X}_t)$. Now 
 \begin{align*}
 	\expec{\tilde{Y}_t \mid \tilde{\scr{F}}_s} & = \int_{\RR^d} \tilde\rho(t,x) \tilde{p}(s,\tilde{X}_s; t,x) \dd x \\ & = \int_{\RR^d} \tilde{p}(s,\tilde{X}_s; t,x) \int_{\RR^{d-m}} \left(\tilde{p}t,x; T, \sum_{j=1}^d \xi_j f_j\right) \dd \xi_{m+1},\cdots, \dd \xi_d \\ & = 
 	\int_{\RR^{d-m}} \tilde{p}\left(s,\tilde{X}_s; T, \sum_{j=1}^d \xi_j f_j\right)\dd \xi_{m+1},\cdots, \dd \xi_d = \tilde\rho(s,\tilde{X}_s)=\tilde{Y}_s.
 \end{align*}
 That is, $(\tilde{Y}_t, \tilde{\scr{F}}_t)$ is a martingale. If $\tilde{\scr{L}}$ denotes the infinitesimal generator of $\tilde{X}_t$, then $\scr{K}=\partial / (\partial t) + \tilde{\scr{L}}$ is the infinitesimal generator of the space time process $(t,\tilde{Y}_t)$. Since $\tilde{Y}_t$ is a martingale, the mapping $(t,x) \mapsto \tilde\rho(t,x)$ is space-time harmonic. Then by Proposition 1.7 in chapter VII of 
 \cite{RevuzYor} 
 $\scr{K} \tilde\rho(t,x)=0$. That is, $\tilde{\rho}(t,x)$ satisfies Kolmogorov's backward equation. 
\end{proof}

This absolute continuity result is only useful for simulating conditioned diffusions if it can be shown to hold in the limit $t\uparrow T$ as well. 
The main line of proof is  the same as in the proof of Theorem 1 in \cite{Schauer}, where at various places $p$ and $\tilde{p}$ need to be replaced with $\rho$ and $\tilde\rho$. However, some of the auxiliary results that are used  require new arguments in the present setting. Moreover, the assumed Aronson type bounds are not suitable for hypo-elliptic diffusions.

\subsection{Proof of Theorem \ref{thm:main}} \label{sec:mainproof-abscont}
We start with introducing some notation.
Define the mapping  $g_\Delta\colon [0,\infty) \times \RR^d \to \RR^m$ by 
\[ g_\Delta(t,x) = \Delta(t) \left( v-\mu(t) - L(t) x\right) \]
and note that $\ZD{t} = g_\Delta(t,X^\circ_t)$. 
 For a diffusion process $Y$ we define the stopping time
\[  \si_k(Y) = T \wedge \inf_{t\in [0,T]} \left\{ \|g_\Delta(t,Y_t)\| \ge k \sqrt{(T-t) \log(1/(T-t))}\right\}, \]
where $k\in \NN$.
We write
\[ \si_k^\circ = \si_k(X^\circ) \qquad \si_k = \si_k(X)\qquad \si_k^\star = \si_k(X^\star). \]

Define $\bar\rho = \tilde\rho(0,x_0)/\rho(0,x_0)$. 
By Proposition \ref{prop-abscont-beforeT} , for any $t<T$ and bounded, $\mathcal F_t$-measurable $f$, we have
\begin{equation}\label{likeliexp}
 \expec{ f(X^\star) \frac{ \tilde  \rho(t, X^\star_t)}{\rho(t, X^\star_t)}} =
\expec{ f(X^\circ)  \bar{\rho} \: \Psi_t(X^\circ)}.
\end{equation}
By taking $f_t(x)=\ind\{t\le \si_k(x)\}$, we get
\begin{equation}\label{eq:llsigma} \bar\rho\, \EE\left[ \Psi_t(X^\circ) \ind\{t\le \si^\circ_k\}\right] = \EE\left[ \frac{\tilde\rho(t,X^\star_t)}{\rho(t,X^\star_t)} \ind\{t\le \si^\star_k\}\right].
\end{equation}
Next, we take $\lim_{k\to\infty} \lim_{t\uparrow T}$ on both sides. We start with the left-hand-side. 
By  Lemma \ref{lem:boundlikelihood}, for each $k\in\NN$,  $\sup_{0\le t\le T}\Psi_t(X^\circ)$ is
uniformly  bounded  on the event $\{T=\si_k^\circ\}$. Hence, by the dominated convergence theorem we obtain 
\[	\lim_{k\to\infty} \lim_{t\uparrow T}\EE\left[ \Psi_t(X^\circ) \ind\{t\le \si^\circ_k\}\right] = \lim_{k\to \infty} \EE\left[ \Psi_T(X^\circ) \ind\{T\le \si^\circ_k\}\right]. \]
Since by definition $\si^\circ_k\le T$, we have $\{T\le \si^\circ_k\}=\{T= \si^\circ_k\}$. Furthermore, 
\[ \ind\{T=\si^\circ_k\} = \ind\left\{\|Z^\circ_{\Delta,t}\| \le k \sqrt{(T-t) \log(1/(T-t))} \right\} \uparrow 1 \quad \text{as}\quad  k \to \infty, \]
by Proposition \ref{thm:limitU}. Therefore, by monotone convergence
\[	\lim_{k\to\infty} \lim_{t\uparrow T}\EE\left[ \Psi_t(X^\circ) \ind\{t\le \si^\circ_k\}\right] =  \EE\left[ \Psi_T(X^\circ)\right]. \]
It remains to show that the right-hand-side of \eqref{eq:llsigma} tends to $1$. We write 
\begin{align*} \rho(0,x_0)  \EE\left[ \frac{\tilde\rho(t,X^\star_t)}{\rho(t,X^\star_t)} \ind\{t\le \si^\star_k\}\right] &= \EE\left[ \tilde\rho(t,X_t) \ind\{t\le \si_k\}\right] \\ &  =\EE\left[ \tilde\rho(t,X_t)\right] - \EE\left[ \tilde\rho(t,X_t) \ind\{t> \si_k\}\right]  
\end{align*}
By Lemma \ref{lem:tildepoverp} the first of the terms on the right-hand-side  tends to $\rho(0,x_0)$ when $t\uparrow T$. The second term tends to zero by Lemma \ref{lem:expec_greater_1}.

To complete the proof we note that by 
 equation (\ref{likeliexp})  and Lemma \ref{lem:tildepoverp} we have  $\bar \rho\, \expec{\Psi_t(X^\circ)} \to 1$ as $t\uparrow T$. 
 In view of the preceding and Scheff\'e's Lemma this implies that 
$\Psi_t(X^\circ) \to \Psi_T(X^\circ)$ in $L^1$-sense as $t \uparrow T$.  
Hence for $s<  T$ and a bounded,  $\mathcal{F}_s$-measurable, continuous functional $g$, 
\[
\expec{g(X^\circ) \bar \rho \Psi_T(X^\circ)} = \lim_{t \uparrow T} \expec{g(X^\star)
\frac{\tilde \rho(t, X^\star_{t})}{\rho(t, X^\star_{t})}}.
\]
By Lemma \ref{lem:tildepoverp} this converges to $\EE\, g (X^\star)$ as $t \uparrow T$ and we find that 
$\EE\, g(X^\circ) \bar \rho \Psi_T(X^\circ) = \EE\, g (X^\star)$.


\begin{lem}\label{lem:boundlikelihood}
Under Assumption \ref{ass:sigma_const} there exists a positive constant $K$ (not depending on $k$) such that
		\[ \Psi_t(X^\circ) \ind_{t\le \si^\circ_m} \le \exp(Kk^2). \]
\end{lem}
\begin{proof}
To bound $\Psi_t(X^\circ)$, we will first rewrite $\scr{G}(s,X^\circ)$ in terms of $\ZD{t}$, $\LD(t)$ and $\MD(t)$, as defined in \eqref{eq:defU} and \eqref{eq:def-tildeC-P}. By display \eqref{eq:tilder-PC}, we have
\[ \tilde{r}(t,X^\circ_t)=\LD(t)^\T \MD(t)\ZD{t}  \quad \text{and} \quad \tilde{H}(t)=\LD(t)^\T \MD(t) \LD(t).\]
Here, the expression for $\tilde{H}(t)$ was obtained from  \[\tilde{H}(t)= -\DD L(t)^\T M(t) (v-\mu(t)-L(t)x)=\DD (L(t)^\T M(t) L(t) x)= L(t)^\T M(t) L(t).\] Hence,
\begin{equation}\label{eq:scrG-rewritten}
\begin{split}\scr{G}(s,X^\circ_s) &= (b(s,X^\circ_s) - \tilde b(s,X^\circ_s))^\T \LD(s)^\T \MD(s)\ZD{s}  \\ &  -  \frac12 \tr\left(\left[a_s - \tilde a(s)\right]  \LD(s)^\T \MD(s) \LD(s)\right)\\ &  +\frac12 \ZD{s}^\T \MD(s) \LD(s)  \left[a_s - \tilde a(s)\right] \LD(s)^\T \MD(s)\ZD{s} .
\end{split}
\end{equation}
On the event $\{t \le \si_k^\circ\}$ we have
\[ \|\ZD{t}\| \le k \sqrt{(T-t) \log(1/(T-t))}. \]
The absolute value of the first term of $\scr{G}$ can be bounded by
\begin{align*}
	 &  \left\| \MD(s)\right\| \left\| \LD(s)\left( \tilde{b}(s,X^\circ_s)-b(s,X^\circ_s) \right)\right\| \|\ZD{s}\| \\& \qquad  \le
  (T-s)^{-1} \|\ZD{s}\|  \\ &\qquad \le c_1 m  (T-s)^{-1/2} \sqrt{\log(1/(T-s))}.
\end{align*}
Here we bounded $\|\MD(t)\| \le \lmax(\MD(t))$, as in \eqref{eq:boundnorm-by-lmax}. 
The absolute value of twice the  second term of $\scr{G}$ can be bounded by 
\[ \tr\left(\LD(s) (a_s-\tilde{a}(s)) \LD(s)^\T \right)  k  \la_{\rm max}(\MD(s)), \]
just as in \eqref{eq:boundtraceterm}. As for a $p\times p$ matrix $A$ we have $\tr(A)\le p \lmax(A) = p \|A\|^2$ (recall we assume the spectral norm on matrices throughout), this can be bounded by 
\[ m\|\LD(s) (a_s-\tilde{a}(s)) \LD(s)^\T\|^2  m  \la_{\rm max}(\MD(s)) \le m^2 c_3 \bar{c}(T-s)^{2\alpha-1} \]
The absolute value of twice the  third term of $\scr{G}$ can be bounded by 
\begin{align*} & \|\ZD{s}\|^2 \|\MD(s)\|^2 \|\LD(s) (a(s)-\tilde{a}(s)) \LD(s)^\T\| \\& \qquad \le k^2 (T-s) \log(1/(T-s)) \bar{c}^2 (T-s)^{-2} c_3 (T-s)^\alpha\\& \qquad  \le k^2 \bar{c}^2 c_3 (T-s)^{\alpha-1} \log(1/(T-s)).
\end{align*}
We conclude that all three terms in $\scr{G}$ are integrable on $[0,T]$. 
	\end{proof}

\begin{lem}\label{lem:lim-ptilde}
For all bounded, continuous $f\colon [0,T] \times \RR^d \to \RR$ 
\[  \lim_{t\uparrow T} \int f(t,x) \tilde{p}(t,x; T,v) \dd x = f(T,v). \]
\end{lem}
\begin{proof}
The proof is just as in Lemma 7 of \cite{Schauer}. 
\end{proof}

\begin{lem}
\label{lem:tildepoverp}
If  Assumption \ref{ass:p<=tildep} holds true, $0<t_1 < t_2 < \dots < t_N< t < T$ and  $g \in C_b(\RR^{Nd})$, then
\[ \lim_{t \uparrow T}  \expec{g(X^\star_{t_1}, \dots, X^\star_{t_N}) \frac{\tilde \rho(t, X^\star_{t})}{ \rho(t, X^\star_{t})}}  = \expec{g(X^\star_{t_1}, \dots, X^\star_{t_N})}.\]
\end{lem}
\begin{proof}
The joint density $q$ of $(X_{t_1}, \dots, X_{t_N})$, conditional on $X_{t_0}=x_0$ is given by
$	q(x_1,\ldots, x_N) = \prod_{i=1}^N p(t_{i-1}, x_{i-1}; t_i, x_i) .$
Hence,
\begin{equation}\label{eq:limitlem}
\begin{split}
& \expec{g(X^\star_{t_1}, \dots, X^\star_{t_N}) \frac{\tilde \rho(t, X^\star_{t})}{ \rho(t, X^\star_{t})}} 
\\ & \qquad = \int g(x_1,\ldots,x_N) \frac{\tilde{ \rho}(t,x)}{ \rho(t,x)} q(x_1,\ldots,x_N) \frac{p(t_N,x_N; t,x)  \rho(t,x)}{ \rho(0,x_0)} \dd x_1\ldots \dd x_N \dd x\\  &  \qquad = \frac{1}{\rho(0,x_0)} \int g(x_1,\ldots,x_N) q(x_1,\ldots, x_N) F(t; t_N, x_N)  \dd x_1\ldots \dd x_N, 
\end{split}	
\end{equation}
where for $t_N<t<T$
\[	F(t; t_N, x_N)= 
 \int p(t_N, x_N; t,x) \tilde{\rho}(t,x) \dd x. \]
We can assume $t\ge (T+t_N)/2$. 
For fixed $t_N$ and $x_N$, the mapping $(t,x) \mapsto p(t_N, x_N; t,x)$ is continuous and bounded, for $t$ bounded away from $t_N$. By Lemma \ref{lem:lim-ptilde} it follows  that $F(t;t_N, x_N) \to \rho(t_N,x_N)$ when $t\uparrow T$. The argument is finished by taking the limit $t\uparrow T$ on both sides of equation \eqref{eq:limitlem}, interchanging limit and integral on the right-hand-side  and noting that the limit on the right-hand-side coincides with $\expec{g(X^\star_{t_1}, \dots, X^\star_{t_N})}$. 

The interchange is permitted by dominated convergence. To see this, first note that $g$ is assumed to be bounded.  Next, 
\[ \int \left(\prod_{i=1}^n p(t_{i-1}, x_{i-1}; t_i, x_i)\right) p(t_N,x_N; t,x) \tilde\rho(t,x) \dd x   \dd x_1\ldots \dd x_N \le C^{N+1} \tilde\rho(t_0,x_0), \]
which follows from repeated application of Assumption \ref{ass:p<=tildep}.

\end{proof}

\begin{lem}\label{lem:expec_greater_1}
Assume that there exists a positive $\delta$ such that $|\Delta(t)| \lesssim (T-t)^{-\delta}$.
If Assumption \ref{ass:p<=tildep} holds true, then
\[ \lim_{k\to \infty}  \lim_{t \uparrow T}\expec{  \tilde\rho(t,X_t) \ind\{t> \si_k\}} .
\]
\end{lem}
\begin{proof}
As in the proof of Lemma 5 in \cite{Schauer}, it suffices to show that 
\begin{equation}\label{eq:tobeshown}   \lim_{k\to \infty} \lim_{t \uparrow T} \expec{ \ind_{\{t >\sigma_k\}}    \int p(\sigma_k,X_{\sigma_k}; t,z) \tilde \rho(t ,z) \dd z } =0 .\end{equation}
Applying Assumption \ref{ass:p<=tildep} and using the Chapman-Kolmogorov relations, we obtain \[  \int p(\sigma_k,X_{\sigma_k}; t,z) \tilde \rho(t ,z) \dd z \le C \tilde\rho(\sigma_k,X_{\sigma_k}).
\]
Define $\tilde{Z}_t=v-\mu(t) - L(t) \tilde{X}_t$. If we denote its transition density by $\tilde{q}$, then \[  \tilde\rho(t,y) = \tilde{q}(t,v-\mu(t)-L(t) y; T, 0), \qquad y \in \RR^d \quad \text{and}\quad t\in [0,T),
 \]
since $\tilde{r}(t,x)$ depends on $x$ only via $L(t)x$.
Define the set
\[ \scr{A}_k = \{(t,y) \in [0,T) \times \RR^d\,:\, \|\Delta(t)(v-\mu(t)-L(t)y)\| = k \eta(t)\}, 
\] where $\eta(t) = \sqrt{(T-t)\log(1/(T-t))}$. Then 
\[ \expec{ \ind_{\{t >\sigma_k\}}    \int p(\sigma_k,X_{\sigma_k}; t,z) \tilde \rho(t ,z) \dd z } \le \expec{\sup_{(t,y) \in \scr{A}_k}\tilde\rho(t,y)},   \]
since by definition of $\si_k$, $\|\Delta(\si_k)(v-\mu(\si_k)-L(\si_k)X_{\si_k})\| = k \eta(\si_k)$. The expectation on the right-hand-side is now superfluous. 
It is easily derived that $\tilde{Z}_t$ satisfies the SDE
\[ \dd \tilde{Z}_t = L(t)\tilde\si(t) \dd W_t \]
and hence  for $x\in \RR^m$
\[ \tilde{q}(t,x; T,0) = \phi_m(0; x, \Md(t)), \]
where we denote the density of the multivariate normal distribution in $\RR^m$ with mean vector $\nu$ and covariance matrix $\Upsilon$, evaluated in $u$ by $\phi_m(u;\nu,\Upsilon)$.  
Hence, stitching the previous derivations together we obtain
\[	\expec{ \ind_{\{t >\sigma_k\}}    \int p(\sigma_k,X_{\sigma_k}; t,z) \tilde \rho(t ,z) \dd z }  \le   \sup_{(t,y) \in \scr{A}_k} \phi_m(0; x, \Md(t)). \]
The right-hand-side, multiplied with $(2\pi)^{m/2}$ equals \[ |M(t)|^{1/2} \exp\left( -\frac12(v-\mu(t)-L(t)y)^\T M(t)(v-\mu(t)-L(t)y)\right)\]
which can be further bounded by
\begin{align*}
&  \sup_{(t,y) \in \scr{A}_k}  |\Delta(t)| |\MD(t)|^{1/2}
\exp\left( -\frac12(v-\mu(t)-L(t)y)^\T \Delta(t) \MD(t) \Delta(t) (v-\mu(t)-L(t)y)\right)\\ 
&\quad \le \sup_{(t,y) \in \scr{A}_k}  |\Delta(t)| ( \lmax(\MD(t))^{m/2} \exp\left(-\frac12\|\Delta(t)(v-\mu(t)-L(t)y))\|^2 \lmin(\MD(t))\right)\\ 
& \quad\le \sup_{t\in [0,T)} |\Delta(t)| \left(\frac{\bar{c}}{T-t}\right)^{m/2} \exp\left(-\frac{\underline{c} k^2\eta(t)^2}{2(T-t)}\right)\\ 
& \quad \lesssim \sup_{t\in [0,T)}  (T-t)^{-\delta-m/2} \exp\left(-\frac{\underline{c} k^2\eta(t)^2}{2(T-t)}\right). 
 \end{align*}
Next, the maximum can be bounded, followed by taking the limit
$k\to \infty$, to see that this tends to zero. This is exactly as in the proof of Lemma 5 in \cite{Schauer}.

\end{proof}

\subsection{Proof of Lemma \ref{lem:verif-p<=tildep}}
By absolute continuity of the laws of $\tilde X$ and $X$ and the abstract Bayes' formula, for bounded $\scr{F}_T$-measurable $f$ we have
\[ \expec{f(X) \mid X_T=v} = \frac{\tilde{p}(0,x_0; T,v)}{p(0,x_0; T,v)} \expec{f(\tilde{X}) \frac{\dd \PP_T}{\dd \tilde{\PP}_T}(\tilde{X} ) \;\middle|\; \tilde{X}_T=v}. \]
Hence, upon taking $f\equiv 1$ and applying Girsanov's theorem we get
\[ p(0,x_0; T,v) = \tilde{p}(0,x_0; T,v) \expec{\exp\left(\int_0^T \eta(\tilde{X}_s)^\T \dd W_s - \frac12 \int_0^T\|\eta(\tilde{X}_s)\|^2 \dd s \right) \middle| \tilde{X}_T=v}.\]
Since $\eta$ is bounded this implies
\[ p(0,x_0; T,v) \propto \tilde{p}(0,x_0; T,v) \expec{\exp\left(\int_0^T \eta(\tilde{X}_s)' \dd W_s  \right) \middle| \tilde{X}_T=v}.\]
Upon defining $\tau_j =\int_0^T \eta_j(\tilde{X}_s)^2 \dd s $, the Dambis-Dubins-Schwarz theorem implies that the expectation on the right-hand-side equals 
\[ \expec{\e^{-\sum_{j=1}^{d'} \int_0^T \eta_j(\tilde{X}_s) \dd W_s^j} \;\middle|\; \tilde{X}_T=v} =
	\expec{\e^{\sum_{j=1}^{d'} {W}_{\tau_j}^j} \;\middle|\;   \tilde{X}_T=v} .\]
By boundedness of $\eta$ there exists constants $\{K_j\}_{j=1}^{d'}$ such that   $\tau_j\le K_j$. Hence the right-hand-side of the preceding display can be bounded by 
	\[	\expec{\e^{\sum_{j=1}^{d'} \sup_{0\le s \le K_j} 
{W}_s^j } \;\middle|\; \tilde{X}_T=v}  =	\expec{\e^{\sum_{j=1}^{d'} \sup_{0\le s \le K_j} {W}_s^j}} = \prod_{j=1}^{d'}  \expec{\e^{\sup_{0\le s \le K_j} {W}_s^j}},
 \]
 where the final equality follows from the components of $\bar{W}$ being independent. The expectation on the right-hand-side is finite, the constant only depending on $T$. To see this: if $B_t$ is a one-dimensional Brownian motion, then $\bar{B}_t =\sup_{0\le s\le t} B_s$ has density $f_{\bar{B}_t}(x)=\sqrt{2/(\pi t)} \e^{-x^2/(2t)}\ind_{[0,\infty)}(x)$, which implies that $\expec{\exp(\bar{B}_t)}<\infty$. 

The statement of the theorem now follows by considering the processes $X$ and $\tilde{X}$ started in $x$ at time $s$ and noting that the  derived constant only depends on $T$.

\section{Discussion}\label{sec:discussion}

\subsection{Extending the approach in \cite{Marchand} to hypo-elliptic diffusions.}

 A potential advantage of the approach in \cite{Marchand} is that, at least in the uniformly elliptic case, there is no matching condition for the diffusion coefficient to be satisfied.  Inspecting the guiding term in \eqref{eq:pull-marchand},  it can be seen  that it is also well defined  when ${\ker }(\si(t,x)^\T L^\T)=\{0\}$, since this assures that the inverse  of $La(t,x) L^\T$ exists for all $t\ge 0$ and $x\in \RR^d$. Unfortunately, this excludes for example the case where the smooth component of an integrated diffusion process is observed (Example \ref{ex:integrated-diffusion-partial1}). Here, the guiding term is given by 
\[ {\rm guid}_1(t,x):= a L(t)^\T \left(\int_t^T L(\tau) \tilde{a} L(\tau)^\T \dd \tau\right)^{-1} (v-L(t)x). \]
   Now it is tempting to adjust the proposals by \cite{Marchand} in Equation \eqref{eq:pull-marchand} in the same way as was done for guided proposals, by replacing $L$ by $L(t)$. This  leads to the guiding term
\[ {\rm guid}_2(t,x):=a L(t)^\T (L(t) a L(t)^\T)^{-1} \frac{v-L(t)x}{T-t}. \]
This guiding term  will  {\it not} give correct bridges though. To see this, if $\underline\beta\equiv 0$ then $X^\circ=X^\star$ but 
\begin{equation}\label{eq:pull1-pull2} {\rm guid}_2(t,x)=\frac13 {\rm guid}_1(t,x) \quad \text{with} \quad  {\rm guid}_1(t,x)=\Bm 0 \\ 3\frac{v-x_1-(T-t)x_2}{(T-t)^2}\Em \end{equation}
(here $x_i$ denotes the $i$-th component of the vector $x$).  We stress that ${\rm guid}_2(t,x)$ has never been proposed in \cite{Marchand} and that the guiding term in Equation \eqref{eq:pull-marchand} is perfectly valid in the uniformly elliptic case. The point we make here, is that it is far from straightforward to generalise \cite{Marchand}'s work to the hypo-elliptic setting. Possibly, the correct generalisation of \cite{Marchand}'s work to the hypo-elliptic case is to take the guiding term of the form
\[ a(t,X^\circ_t) L(t)^\T \left(\int_t^T L(\tau) a(\tau,X^\circ_\tau) L(\tau)^\T \dd \tau\right)^{-1} (v-L(t)X^\circ_t). \] This term is however  unattractive from a computational point of view.

\subsection{State-dependent diffusion coefficient}
We have formulated our results for  state-dependent diffusion coefficients $\si$.  The main difficulty however resides in checking the fourth inequality of Assumption \ref{ass:sigma_const}. We conjecture that the ``right'' way to deal with this term is to bound 
\[ \| \LD(t) (\tilde{a}(t)-a(t,X^\circ_t)) \LD(t)^\T\| \lesssim \|\ZD{t}\|. \]
Then the final term in inequality \eqref{eq:bound-ZD}
would be replaced with $C_4 \int_{t_0}^t (T-s)^{-2} \|\ZD{s}\|^3 \dd s$. The conjecture is motivated by the  proof of Theorem 2 in \cite{Schauer}. Obtaining such an inequality is not straightforward, as is the corresponding Gronwall type argument. We postpone such investigations to future research.

\section*{Acknowledgement}
We thank O.\ Papaspiliopoulos (Universitat Pompeu Fabra Barcelona),    S.\ Sommer (University of Copenhagen) and M.\ Mider (University of Warwick) for stimulating discussions on diffusion bridge simulation.

\appendix

\section{Existence of $\tilde{r}$ if $L=I$}\label{sec:app-control}
In case $L=I$, the existence problem of transition densities has  been studied in control theory as well. 
\begin{defn}
The pair  $(\tilde{B}, \tilde{\si})$ is called  {\it completely controllable} at $s$ if, for any $t>s$ and $x, y \in \RR^d$, there exists a function $v \in L^2[s,t]$ and corresponding solution $Y$ of  
\[ d Y(u) = \left( \tilde{B}(u) Y(u) + \tilde{\si}(u) v(u) \right) d u, \qquad Y(s)=x\] such that $Y(t)=y$.
\end{defn}
The following lemma is proved in  \cite{Hermes-LaSalle}.
\begin{lem}
The following are equivalent:	
\begin{enumerate}
\item $(\tilde{B},\tilde\si)$ is completely controllable at $s$;
\item Non degenerate gaussian transition densities $\tilde{p}(s,x; t,y)$ exist;
\item For arbitrary gaussian initial data $\tilde{X}_s$, the random vector $\tilde{X}_t$ is non degenerate Gaussian for $t>s$. 
\end{enumerate}
\end{lem}

If  $\tilde{B}$, $\tilde\si$ in \eqref{eq:auxlin} are constant matrices, complete controllability is equivalent to $\mbox{rank}(C)=d$, where the  controllability matrix $C$ is defined by 
\[ C:= [\tilde\si,\: \tilde{B}\tilde\si,\ldots, \tilde{B}^{d-1} \tilde
\si ]\]
(see page 74 in \cite{Hermes-LaSalle} or proposition 6.5 in \cite{Karatzas-Shreve}). 
This provides an easily verifiable condition for complete controllability.

\section{Gronwall type inequality}

In the proof of Theorem \ref{thm:limitU} we used the following Gronwall type inequality:

\begin{lem}\label{lem:gronwall2}
Assume $t\mapsto \zeta(t)$ is continuously differentiable and nonnegative on $[t_0,t_1)$. Assume $t\mapsto f_1(t)$ and $t\mapsto f_2(t)$ are continuous and nonnegative on $[t_0,t_1)$. Suppose $t\mapsto u(t)$ is a continuous and nonnegative function on $[t_0, t_1)$ satisfying the inequality
\[ u(t) \le \zeta(t)+ \int_{t_0}^t f_1(s) \sqrt{u(s)} \dd s +\int_{t_0}^t f_2(s) u(s) \dd s  , \qquad t\in [t_0, t_1). \]
 Then
\[ u(t) \le  \left(\sqrt{\zeta(t_0) + \int_{t_0}^t |\zeta'(s)| \dd s}+ \frac12\int_{t_0}^t f_1(s) \dd s \right)^2 \exp\left(\int_{t_0}^t f_2(s) \dd s\right). \]
\end{lem}
\begin{proof}
	This is a special case of theorem 2.1 in \cite{AGARWAL2005}. In their notation, we have $n=2$, $w_1(x)=\sqrt{x}$, $W_1(x)=2\sqrt{x}$ (taking  $u_1=0$), $w_2(x)=x$, $W_2(x)=\log x$ (taking $u_2=1$). 
\end{proof}


\section{Hypoellipticity}
\label{sec:hormander}
\begin{prop}
\label{prop:hormander}
Consider the diffusion~\eqref{eq:sde}
with
$ b(t,x) = B x + \beta(t,x)$ for $B \in \RR^{d \times d}$ and $\beta \in C^{\infty}([0,T] \times \R^d, \R^d)$, and with $\sigma \in C^{\infty}([0,T] \times \R^d, \R^{d \times d'})$.
Suppose that, for all $(t,x) \in (0,T) \times \R^d$, the pair  $(B, \sigma(t,x))$ is controllable, i.e. the rank of the matrix concatenation
\[ \begin{bmatrix} \sigma(t,x) & B \sigma(t,x) & \dots & B^{d-1} \sigma(t,x) \end{bmatrix}\] 
is equal to $d$. Furthermore suppose that for all $(t,x) \in (0,T) \times \R^d$ and all tuples $(n_0, n_1, \dots, n_d) \in \{0, 1, \dots, d-1\}^{d+1}$,
\[  (\partial_{t})^{n_0} \prod_{i=1}^d (\partial_{x_i})^{n_i} \beta(t,x) \in \Col \sigma(t,x)\] and 
\[ \Col \left( (\partial_{t})^{n_0} \prod_{i=1}^d (\partial_{x_i})^{n_i} \sigma(t,x) \right) \subset \Col \sigma(t,x),\] i.e. the column spaces of all partial derivatives of $\beta(t,x)$ and $\sigma(t,x)$, including $\beta(t,x)$ itself, belong to the column space of $\sigma(t,x)$.
Finally suppose there exists at most one strong solution to~\eqref{eq:sde} (which is the case if e.g. $\beta$ and $\sigma$ satisfy a linear growth condition).
Then for all initial conditions $x_0$ and all $t \geq 0$, the distribution of $X_t$ admits a density function $p(t,x,y)$:
\[ \EE_{x_0}[f(X_t)] = \int_{\R^d} p(t,x_0,y) f(y) \, d y, \quad f \in C_0(\R^d),\]
and $p$ is a smooth (infinitely often continuously differentiable) function on $(0,\infty) \times \R^d \times \R^d$.
\end{prop}

\begin{proof} 
Write $(\sigma_j)_{j=1}^{d'}$ for the columns of $\sigma$ so that
\[ \sigma(t,x) = \begin{bmatrix} \sigma_1(t,x) & \dots & \sigma_{d'}(t,x) \end{bmatrix}.\]
The Stratonovich form of~\eqref{eq:sde} is given by
\[ \dd X_t = \widetilde b(t,X_t) \dd t + \si(t,X_t) \circ \dd W_t,\qquad X_0=x_0,\qquad t\in [0,T],\] 
where
$\widetilde b(t,x) = B x + \widetilde \beta(t,x)$ with coordinates of $\widetilde \beta$ given by \[ \widetilde \beta^i(t,x) = \beta^i(t,x) - \tfrac 1 2 \sum_{j=1}^{d'} \sum_{l=1}^d \sigma_j^l (\partial_l \sigma^i_j)(t,x).\]
Observe that $\widetilde \beta(t,x) \in \Col \sigma(t,x)$, just like $\beta(t,x)$.

In particular, the generator of the diffusion~\eqref{eq:sde} can be given in terms of the first order differential operators
\[ \mathcal A_0 f(t,x) = \partial_t f(t,x) + \langle \widetilde b(t,x), \nabla_x f(t,x) \rangle, \quad \mathcal A_j f(t,x) = \langle \sigma_j(t,x), \nabla_x f(t,x) \rangle, \quad j = 1, \dots, d', \]
as $\mathcal L = \mathcal A_0 + \tfrac 1 2 \sum_{j=1}^{d'} \mathcal A_j^2$. In this proof we will use  without further comment (i) Einstein's summation convention and (ii) the canonical identification of first order partial differential operators $\mathcal A = a^i \partial_i = a^0(t,x) \partial_t + \sum_{i=1}^d a^i(x) \partial_{x_i}$ (acting on functions $f : [0,\infty) \times \R^d \rightarrow \R$) with vector fields $\begin{bmatrix} a^0(t,x) & a^1(t,x) & \dots & a^d(t,x) \end{bmatrix}^T \in C^{\infty} ([0,\infty) \times \R^d;\R \times \R^d)$ without further comment. 
The commutator $[\mathcal U_1, \mathcal U_2]$ of two vector fields $\mathcal U_1$, $\mathcal U_2$ is as usual defined by 
\[ [ \mathcal U_1, \mathcal U_2] f(t,x) = \mathcal U_1 \mathcal U_2 f(t,x) - \mathcal U_2 \mathcal U_1 f(t,x).\]
For $l = 0, \dots, d-1$, write 
\[ V^l :=\Col \begin{bmatrix} \sigma(t,x) & B \sigma(t,x) & \dots & B^l \sigma(t,x) \end{bmatrix}.\]
Write $[\cdot, \mathcal A_0]^l$ for taking the Lie bracket with $\mathcal A_0$ repeatedly, i.e. recursively we define 
\[  [\mathcal U,\mathcal A_0]^0 f = \mathcal U f \quad \mbox{and} \quad [\mathcal U, \mathcal A_0]^{l+1} = \left[ [\mathcal U, \mathcal A_0]^l, \mathcal A_0 \right], \quad l=0,1,2, \dots.\]
We first compute
\begin{align*} [\mathcal A_j, \mathcal A_0] f & = \sigma_j^k B_k^i \partial_i f -(\partial_t \sigma_j^k)\partial_k f - B^i_l x^l (\partial_i \sigma_j^k) \partial_k f   - \widetilde \beta^i(\partial_i \sigma_j^k)\partial_k f + \sigma_j^k (\partial_k \widetilde \beta^i) \partial_i f.
\end{align*}
Observe that the first term represents the operator $\langle B \sigma_j, \nabla f\rangle$, and the remaining terms assume values in $V^0 = \Col \sigma(t,x)$.
By iterating we obtain $[\mathcal A_j, \mathcal A_0]^l = B^{l} \sigma + \mathcal U$, where $\mathcal U(t,x) \in V^{l-1}$ for all $(t,x)$.
By the controllability assumption on $B$ and $\sigma$, the vectors
\[ \{ [\mathcal A_j, \mathcal A_0]^l(t,x) : l = 1, \dots, d-1,  j = 1, \dots, d' \}  \]
span $\R^d$ for all $(t,x)$. Adding $\mathcal A_0$ to the collection of vectors gives that
\[ \operatorname{span} \{ \mathcal A_0, [\mathcal A_j, \mathcal A_0]^l(t,x) : l = 1, \dots, d-1, j = 1, \dots, d' \}  \]
has dimension $d+1$, for all $(t,x) \in (0,T) \times \R^d$.
The result now follows from H\"ormander's theorem lifted to $(0,T) \times \R^d$, e.g. \cite[Corollary 5.8]{Williams1981}.
\end{proof}

\section{Derivation of the conditioned process}\label{sec:Doob-h}
The SDE for the conditioned process, given in \eqref{eq:xstar}, can be derived using Doob's $h$-transform. 
\begin{ass}\label{ass:smooth-rho}
The mapping $ \rho \colon \RR_+ \times \RR^d \to \RR$ is $C^{1,2}$ and strictly positive. 
\end{ass}
Suppose  $0\le s<t<T$. By the Chapman-Kolmogorov equations, for a compactly supported $C^\infty$-function $f\colon \RR^d \to \RR$ we have
\[ \EE\left[ f(X_t) \mid X_s=x, LX_T=v\right] = \int f(y) p(s,x; t,y) \frac{\rho(t,y)}{\rho(s,x)}\dd y. \]
Define $g(t,x)=f(x) \rho(t,x)$. Using the preceding display we find that the infinitesimal generator of the conditioned process, say $\scr{L}^\star$, equals
\begin{align*} \scr{L}^\star f(x)&=\lim_{\Delta \downarrow 0} \Delta^{-1}\left( \EE\left[ f(X_{s+\Delta}) \mid X_s=x, LX_T=v\right] - f(x) \right)\\&=\frac1{\rho(s,x)} \lim_{\Delta \downarrow 0} \Delta^{-1}\left( \int g(s+\Delta,y) p(s,x;s+\Delta,y) \dd y -g(s,x)
\right)\\&=\frac1{\rho(s,x)} \lim_{\Delta \downarrow 0} \Delta^{-1}\left( \EE \left[ g(s+\Delta,X_{s+\Delta}) \mid X_s=x\right] -g(s,x)
\right).
\end{align*}
By Assumption \ref{ass:smooth-rho}, $g$ is a compactly supported $C^\infty$-function in the domain of the infinitesimal generator  $\scr{K}$ of the space -time process $(t,X_t)$. Therefore,
\[ \scr{L}^\star f(x) = \frac1{\rho(s,x)} \left(\scr{K}g\right)(s,x), \] 
where 
\[ \scr{K} \phi(s,x)= \frac{\partial}{\partial s} \phi(s,x) + \sum_i b_i(s,x) \DD_i \phi(s,x) +\frac12 \sum_{i,j} a_{ij}(s,x) \DD^2_{ij} \phi(s,x). \]
Here (and in the following) all summations run over $1,\ldots, d$, $\DD_i=\frac{\partial}{\partial x_i}$ and 
$\DD^2_{ij}=\frac{\partial^2}{\partial x_i\partial x_j}$. 
Using the definition of $g$ we get 
\begin{align*}
\scr{L}^\star f(x)&=  \sum_i\left(b_i(s,x) + \sum_j a_{ij}(s,x) \frac{\DD_j \rho(s,x)}{\rho(s,x)}\right) \DD_i f(x) \\ & \quad + \frac12 \sum_{i,j} a_{ij}(s,x) \DD^2_{ij} f(x)+\frac{f(x)}{\rho(s,x)} \left(\scr{K}\rho\right)(s,x) . 
\end{align*}
We claim $\left(\scr{K}\rho\right)(s,x)=0$ (that is, $\rho(t,x)$ satisfies Kolmogorov's backward equation). The drift and diffusion coefficients of the conditioned process can then be identified from the infinitesimal generator $\scr{L}^\star$. 
To verify the claim, first note that $Z_t=\rho(t,X_t)$ defines a martingale: if $\scr{F}_s$ is the natural filtration of $X$, then
\[
	\EE [Z_t \mid \scr{F}_s]  = \int p(s,X_s; t,x) \rho(t,x) \dd x =Z_s
\]
where we used the Chapman-Kolmogorov equations. Therefore, $(t,x)\mapsto \rho(t,x)$ is space-time harmonic and then the claim follows from  proposition 1.7 of Chapter VII in \cite{RevuzYor}.

\bibliographystyle{harry}

\bibliography{lit}

\end{document}